\documentclass[a4paper]{article}

\usepackage{fullpage} 
\usepackage[utf8]{inputenc}
\usepackage{subfigure}
\usepackage{amsfonts,amssymb,amsmath,dsfont,amsthm}
\usepackage{enumitem}
\usepackage{xspace}
\usepackage{graphicx}
\usepackage[colorlinks=true,citecolor=blue]{hyperref}
\usepackage{bm}
\usepackage{color}
\usepackage{comment}

\newcommand{\correct}[1]{\textcolor{blue}{#1}}

\newtheorem{theorem}{Theorem}
\newtheorem{lemma}{Lemma}
\newtheorem{corollary}{Corollary}
\newtheorem{proposition}{Proposition}

\theoremstyle{remark}
\newtheorem{rem}{Remark}

\newcommand{\proofparagraph}[1]{~\\ \noindent \textbf{#1}}

\newcommand{\proba}{\mathds{P}}
\newcommand{\esp}{\mathds{E}}
\newcommand{\bigO}{\mathcal{O}}
\newcommand{\exactbigO}{\Theta}
\newcommand{\smallo}{o}

\newcommand{\eps}{\varepsilon}

\newcommand{\mA}{\mathcal{A}}
\newcommand{\mB}{\mathcal{B}}
\newcommand{\mC}{\mathcal{C}}

\newcommand{\mF}{\mathcal{F}}
\newcommand{\mG}{\mathcal{G}}
\newcommand{\mH}{\mathcal{H}}
\newcommand{\mK}{\mathcal{K}}

\newcommand{\mSG}{\mathcal{SG}}
\newcommand{\mMG}{\mathcal{MG}}
\newcommand{\mPatch}[1]{\mathcal{P}^{#1}}
\newcommand{\de}{\,\mathrm{d}}

\newcommand{\vect}[1]{\bm{#1}}
\newcommand{\vdelta}{\vect{\delta}}

\newcommand{\vy}{\vect{y}}

\newcommand*{\ie}{\textit{i.e.},\@\xspace}

\newcommand*{\resp}{resp.\@\xspace}
\newcommand{\aut}[1]{|\operatorname{Aut}(#1)|}

\newcommand{\sg}{\operatorname{SG}}
\newcommand{\mg}{\operatorname{MG}}
\newcommand{\Patch}[1]{\operatorname{Patch}_{#1}}

\newcommand{\disjointpatch}[1]{\operatorname{Disj}\mathcal{P}^{#1}}

\newcommand{\disjointmg}{\operatorname{DisjMG}}

\newcommand{\jointmg}{\operatorname{NonDisjMG}}

\newcommand{\disjointmMG}{\operatorname{Disj}\mMG}

\newcommand{\jointmMG}{\operatorname{NonDisj}\mMG}

\newcommand{\pair}{\operatorname{Pair}}

\newcommand{\support}{\operatorname{Support}}

\title{
Threshold functions for small subgraphs in simple graphs and multigraphs}

\author{
Gwendal Collet\footnote{Institute of Discrete Mathematics and Geometry, TU Wien, Wiedner Hauptstr. 8--10/104, 1040 Wien, Austria. Supported by the Austrian Science Foundation FWF, grant SFB F50-02.}  
\\
\'Elie de Panafieu\footnote{Nokia--Bell Labs and LINCS (France). This work was partially founded by the Austrian Science Fund (FWF) grant F5004, the Amadeus program and the PEPS HYDrATA.}
\\
Dani\`ele Gardy\footnote{DAVID Laboratory, University of Versailles Saint Quentin (France). Partially supported by the Amadeus project 33697ZK \emph{Threshold problems and phase transitions in graph-like structures} (2015--16) and by the ANR-MOST project MetaConc (2015--19).}  
\\
Bernhard Gittenberger\footnote{Institute of Discrete Mathematics and Geometry, TU Wien, Wiedner Hauptstr. 8--10/104, 1040 Wien, Austria. Supported by the Austrian Science Foundation FWF, grant SFB F50-03 and the \"OAD grant Amadée F01/2015.} 
\\
Vlady Ravelomanana\footnote{IRIF, University of Paris 7 (France). Partially supported by the Amadeus project 33697ZK \emph{Threshold problems and phase transitions in graph-like structures} (2015--16), by the project \emph{Combinatorics in Paris} (2014--17)  and by the CNRS-PICS project \emph{Constraint analysis through analytic combinatorics} (2017--19).} }

\begin{document}
\maketitle

\begin{abstract}
We revisit the problem of counting the number of copies of a fixed graph in a random graph or multigraph, for various models of random (multi)graphs.
For our proofs we introduce the notion of \emph{patchworks} to describe the possible overlappings of copies of subgraphs. 
Furthermore, the proofs are based on analytic combinatorics to carry out asymptotic computations.
The flexibility of our approach allows us to tackle a wide range of problems. We obtain the asymptotic number and the limiting distribution of the number of subgraphs which are isomorphic to a graph from a given set of graphs. The results apply to multigraphs as well as to (multi)graphs with degree constraints. 
One application is to scale-free multigraphs, where the degree distribution follows a power law, for which we show how to obtain the asymptotic number of copies of a given subgraph and give as an illustration the expected number of small cycles.

\end{abstract}



		\section{Introduction}
        

Since the introduction of the random graphs $G(n,m)$ and $G(n,p)$ 
by Erd\H os-R\'enyi~\cite{ER60} in 1960 one of the most studied parameters 
is the number of subgraphs isomorphic to a given graph $F$. 

Throughout the paper, for a given graph or subgraph $G$, $E(G)$ (resp. $V(G)$) denotes the set of its edges (resp. vertices).
For a given graph $F$ denote by $G[F]$ the number of copies of $F$ contained in the random graph $G(n,p)$ as a subgraph. 
Observe that by the asymptotic equivalence between $G(n,p)$ and $G(n,m)$ (see for instance~\cite{JaLuRu00}) results from one model can be translated into the other rigorously.
The distribution of $G[F]$ has been studied extensively since the seminal work of Erd\H os-R\'enyi \cite{ER60} who gave the first results in this direction. 
A general threshold for $\{G[F] > 0\}$ has been established in 1981 by Bollob\'as~\cite{Bo81} located at $p=n^{-1/\mu}$ where $\mu = \max_H\{|E(H)|/|V(H)|: H \subset F\}$. If the order of magnitude of $p$ is smaller than the threshold, then asymptotically almost surely there is no subgraph $F$ in 

A major reference about the distribution of $G[F]$ is the work of Ruci{\'n}ski~\cite{Ru88} stating that the number of copies of $F$ is asymptotically normal if and only if $np^{\mu} \rightarrow \infty$ and $n^2(1-p)\rightarrow \infty$. 
In the same paper, Ruci{\'n}ski proved also that at the threshold the number of subgraphs $F$ follows a Poisson law if and only if $F$ is strictly balanced.

In this context, Janson, Oleszkiewicz and Ruci\'nski~\cite{JaOlRu04} developed a moment-based method that gives estimates for $\mathbb{P}( G[F] \geq (1+\varepsilon) \mathbb{E}G[F])$ which are best possible up to logarithmic factors in the exponent (the authors proved that their exponential bounds on the upper tail of $G[F]$ are tight).
It is important to remark that the notion of strongly balanced graphs, introduced by Ruci{\'n}ski and Vince~\cite{RuVi86}, plays a key role to obtain the results mentioned above.

Recently, there has been an increasing interest in the study of constrained random graphs such as random graphs with given degree sequences or random regular graphs.
In these directions, the number of given subgraphs in such structures has been also studied. 
For instance, Wormald~\cite{wormald-survey} proved in his survey about random regular graphs that the number of short cycles in these structures follows asymptotically a Poisson distribution.
McKay, Wormald and Wysocka~\cite{McKayWormaldWysocka} consider random regular graphs of degree $d$ and show that, when $g$ is such that $(d-1)^{2g-1} = o(n)$, the numbers of cycles of length up to~$g$ are asymptotically distributed as independent Poisson variables. 
Kim, Sudakov and Vu ~\cite{KiSuVu07}, considering a regular unlabeled graph with $n$ vertices of degree~$d$ and a fixed subgraph $H$, study how the probability that a copy of $H$ occurs varies when $d$ grows, show that this probability gets close to~1 for $d$ around $n^{1 - 1/E(H)}$ and prove the convergence of the number of copies of $H$ towards a Poisson distribution.
Gao and Wormald~\cite{GaoWormald08} prove the asymptotic normality of the number of copies of a  strictly balanced subgraph $H$ in a random  $d$-regular graph when $d$ grows large.

Instead of constraining the whole graph, we can also require that the \emph{subgraph} is regular.
An article by Bollob\'as, Kim and  Verstra\"{e}te~\cite{BoKimVe06} considers the appearance of such a $k$-regular subgraph when the density of the graph is around~$4k$.
Several papers have studied the relation between a $k$-regular subgraph and a $k$-core:
Pra\l at, Verstra\"{e}te and Wormald~\cite{PrVeWo11} prove that the threshold for the appearance
of a $k$-regular subgraph is at most the threshold for the appearance of a non-empty $(k+2)$-core,
a result improved first by Chan and Molloy~\cite{ChanMolloy12}  to a $(k+1)$-core, and further by
Gao~\cite{Gao14}, who showed that the size of a $k$-regular subgraph is ``close'' to the size of the $k$-core.
Concurrently, Letzter~{\cite{Letzter13}} has obtained the existence of a sharp threshold for the existence of a $k$-regular subgraph for $k \geq 3$.
Very recently, Mitsche, Molloy and Pra\l at~\cite{MiMoPr18} proved that a random graph $G(n, p=c/n)$ typically has a $k$-regular subgraph if $c > e^{- \Theta(k)}$, which is above the threshold for the appearance of a $k$-core.

Now extend regular graphs and consider graphs whose degrees form a specified degree sequence.
An early reference on the enumeration of such graphs is the article of McKay and Wormald~\cite{McKayWormald90}, followed by 
Greenhill and McKay~\cite{GreenhillMcKay13} who studied the asymptotic number of sparse multigraphs with degree sequences.
Barvinok~\cite{Barvinok10} studies directed and bipartite graphs with prescribed degree, and two  papers by Barvinok and Hartigan~\cite{BarvinokHartigan12,BarvinokHartigan13} establish results about (uniform random) graphs with a given degree sequence. In~\cite{BarvinokHartigan12}, the authors count asymptotically the number of $m \times n$ matrices with prescribed row and column sums, so that their results can be applied to the number of graphs and bipartite graphs with prescribed degrees (on both sides); in~{\cite{BarvinokHartigan13}} they obtain the number of labeled graphs where the degree sequence is fixed.
More recently, Gao and Wormald~{\cite{GaoWormald16}} consider sparse graphs and present a survey of enumeration results for graphs with given degree sequences as well.

If we are interested in the appearance of subgraphs in graphs with specified degree sequences, a good survey of the results up to 2010 is by McKay~{\cite{BmcKay-ICM2010}}.
McKay~\cite{McKay11} again studies the  structure of a random graph with a given degree sequence, including the probability of a given subgraph or induced subgraph.
Chatterjee, Diaconis and Sly~{\cite{CDS11}} consider a general model for (dense) graphs with a given degree sequence and the existence of a limit for  sequence of such graphs; this allows them to obtain a general formula from which one might deduce results on the number of triangles (although not explicitly given).
Very recent results by Greenhill et al.~{\cite{GreenhillIsaevMcKay18}} give the asymptotic  expected number of copies of a graph and of induced subgraphs in  a random graph with a known degree sequence.
As multigraphs model many real-world networks, subgraphs counts have
also been derived for random multigraph models with prescribed degrees.
For very recent works in these directions we refer to the preprint of Angel, van der Hofstad
and Holmgren~\cite{AvdH16}
 where the authors consider multigraphs with prescribed degrees
 and study Poisson approximations of the number of self-loops and multiple
 edges as well as  an estimate on the total variation distance
between the number of self-loops and multiple edges and
the Poisson limit of their sum. Barbour and R\"ollin~\cite{BR17} provide a general normal approximation theorem
for local graph statistics in the same model.

The next step after fixing the degree sequence is allowing this sequence to follow some probability distribution.
An important class of graphs with such a distribution is that of scale-free graphs, \ie  graphs where the degree distribution follows a power law which means that the probability that a vertex has degree $d$ is proportional to $d^{-\gamma}$ for some $\gamma$.
Results from the afore-mentioned article by Gao and Wormald~{\cite{GaoWormald16}} can be applied to some power-law sequences. 
Van der Hofstad~\cite{vdHof16} gives a nice presentation of the different models for graphs; see also his  survey~\cite{Hofstad17} on the configuration model.

Van der Hofstad, Janssen, van Leeuwaarden and Stegehuis~\cite{SHJL17,HJLS17} consider triangles, or rather the \emph{clustering coefficient,} in a class of simple graphs with a hidden variables model and a power-law degree distribution.
Van der Hofstad, van Leeuwaarden and Stegehuis~\cite{HLS17a} consider the number of occurrences of a small connected graph, either as a subgraph or as an induced subgraph; all their results are for the so-called ``erased'' configuration model, which amounts to a simple graph model. In a companion paper~\cite{HLS17b} they consider clustering, \ie the probability of existence of an edge between two neighbours of a given vertex in the configuration model. When the degree of the vertex becomes at least of order $\sqrt{n}$, this probability becomes that of a power law. Their result can be used to derive the expected number of triangles, when a triangle with multiple edges is counted once (this is again the erased configuration model).

\medskip
Our goal in this paper is to revisit some of  these results and to extend them, through analytic combinatorics and extensive use of generating functions for counting graphs with a specified subgraph, or with a given number of subgraphs.

Ours is not the first paper that approaches graph problems with these tools.
Roughly at the same time as the pioneer articles of Flajolet, Knuth and Pittel~\cite{FKP89} about the appearance of cycles and of Janson \emph{et al.}~\cite{JKLP93} about the birth of the giant component, higher-dimensional multivariate generating functions where variables are associated to vertices of the graph were used by McKay and Wormald~\cite{McKayWormald90}. 
Such highly multidimensional generating functions appear again in further papers, see McKay~\cite{BmcKay-ICM2010,McKay11} and Barvinok and Hartigan~\cite{BarvinokHartigan13}.
An important development  was the study of planar graphs by Gimenez and Noy~\cite{GimenezNoy09} through analytic combinatorics, followed by several papers in the same direction.
E.g., the recent paper by Drmota, Ramos and Ru\'e~\cite{DrmotaRamosRue} deals with the limiting distribution of the number of copies of a subgraph in subcritical graphs. Noy, R\'equil\'e and Ru\'e~\cite{NRR18} study precise properties of random cubic planar graphs including, most notably for the topics of this article, a proof of asymptotic normality for the number of triangles.
Another relevant result, upon which we shall build Section~\ref{sec:results-degree-constraints}, is the enumeration of graphs whose degrees must belong to a specific set, presented by de Panafieu and Ramos~\cite{EdPR16}.

\subsubsection*{Overview of results}

In the next section we give the formal definitions of our model and the objects we are interested in: simple graphs and multigraphs, possibly weighted.

Section~\ref{sec:general_weights} presents some combinatorial results on the expected number of subgraphs that can be obtained without resorting to analytic combinatorics tools.
We give here the expected number of subgraphs belonging to a given family~$\mF$ in simple graphs or in multigraphs.
The notion of \emph{patchwork} of copies of subgraphs, which is defined there, allows us to study the distribution of the number of occurrences of a subgraph.
We are then able to consider the total number (weight) of simple graphs or of  multigraphs with a specific number of occurrences of subgraphs in~$\mF$.
We finally derive a Poisson limiting distribution for the number of occurrences of a strictly balanced subgraph in a weighted multigraph.

The tools from analytic combinatorics which we shall use in the rest of the article are generating functions enumerating families of graphs. They are presented in Section~\ref{sec:fgs-graphs-multigraphs}; we also give there the first generating functions for the families we consider.

The following sections are devoted to probabilistic results under two different random models. In Section~\ref{sec:subgraphs}, (multi)graphs are chosen uniformly at random among all graphs of the same size (number of vertices and edges): This is reminiscent of the Erd\H{o}s-R\'enyi $G(n,m)$ model. In Sections~\ref{sec:results-degree-constraints} and \ref{sec:failed-assumptions}, we consider weighted (multi)graphs to study different degree distributions. Both distributions can be achieved by the means of Boltzmann samplers~\cite{DFLS04}, presented in more detail in Section~\ref{sec:config-Boltzmann}, which automatically translate our combinatorial decomposition of the (multi)graphs (weighted or not) into a random sampler.
In particular, the Boltzmann sampler for multigraphs weighted according to their degrees
produces multigraphs following the same distribution as the configuration model does.
This equivalence, detailed in Section~\ref{sec:config-Boltzmann},
bridges the gap between analytic combinatorics and the ``pure'' probabilistic setting.

We address the problem of evaluating the number of copies of a given subgraph in Section~\ref{sec:subgraphs}, for both simple graphs and multigraphs.
Theorems~\ref{th:distinguished_multi} and~ \ref{th:distinguished_simple} give exact and asymptotic expressions for the number of multigraphs and simple graphs with one distinguished subgraph in~$\mF$; then Theorems~\ref{th:exact_mgF} and~\ref{th:subgraphs} give the number of multigraphs and simple graphs with exactly $t$ subgraphs in~$\mF$.
The probability that there is at least one copy of a subgraph of~$\mF$ with high density goes to~0, as shown in Corollary~\ref{th:general_small_density} for multigraphs when $m = o(n^{2 - 1/d(F)})$;
Corollary~\ref{cor:essential_density_multigraph} is a more precise variant of this result when we know the essential density of the subgraph.
Corollary \ref{cor:prob-densest-simple} is the equivalent result for simple graphs; it gives a new proof of the upper bound on the average number of copies of a densest subgraph in a simple graph for $m = \bigO(n^\alpha)$, $\alpha$ fixed $<2$, and shows that the number of copies tends almost surely to~0 when $\alpha < 2 - 1/d^* (F)$, with $d^*(f)$   the essential density of the subgraph~$F$.
Finally, Theorems~\ref{th:limit_law_multi} and~\ref{th:limit_law_simple} prove a Poisson distribution for the number of copies of a strictly balanced subgraph $F$ for a random multigraph and for a random simple graph, respectively, in the range $m \sim c n^{2-1/d(F)}$ where $d(F)$ is the density of the subgraph~$F$.

Section~\ref{sec:results-degree-constraints} considers how to extend  these results to  multigraphs with degree constraints. 
We present our model of randomness for simple graphs or multigraphs with degree constraints in Section~\ref{sec:degree-constraints} and examine its relation to the well-known \emph{configuration model} in Section~\ref{sec:config-Boltzmann}.
Theorem~\ref{th:distinguished_degree_constraints-symbolic} is the analog of Theorems~\ref{th:distinguished_multi} and~\ref{th:distinguished_simple} for weighted graphs; it gives an exact expression for the total weight of multigraphs with one distinguished subgraph in the family~$\mF$, from which we derive the expected number of subgraphs belonging to~$\mF$ and the probability that there are $t$ such subgraphs in Corollaries~\ref{th:deg_constraints_mean} and~\ref{th:deg_constraints_exact_subgraphs}.
To get more precise results, we have to consider properties of the set of weights.
Weight sets with only finitely many nonzero elements are considered in Section~\ref{sec:finite-set-weights} for $n$ and $m$ proportional. There we show that the only subgraphs that have a nonzero probability are trees and unicycles, derive the expected number of copies of a tree, and prove a Poisson limiting distribution for the number of occurrences of cycles of given length (Theorem~\ref{th:finite_trees_cycles}).
We first obtain a general result when the weights do not grow too fast (in a sense that we make precise), then derive the weighted equivalent of Theorems~\ref{th:limit_law_multi} and~\ref{th:limit_law_simple} (this is Theorem~\ref{th:limit_multi_weights}).

We consider further examples of weight sets in Section~\ref{sec:failed-assumptions}, most notably quickly increasing weights.
This occurs for instance in the important case of \emph{power-law multigraphs}, also known as \emph{scale-free networks}, which are treated in Section~\ref{sec:power-law}.
We sketch there a general approach to finding the expected number of copies of a sub-multigraph in a multigraph and give a complete answer for small cycles in Theorem~\ref{th:power-law-cycles}.
We also get results on the threshold for the appearance of trees in the case where $m/n \rightarrow 0$ in Section~\ref{sec:failed-trees} (this is Theorem~\ref{th:weights-trees}), and on multigraphs where the set of vertex degrees is periodic in Section~\ref{sec:failed-regular}.


		\section{Models and definitions}\label{sec_Models}
        

Most of the following definitions come from Erd\H os-R\'enyi~\cite{ER60} and Bollob\'as~\cite{Bo81} for graphs, and from Flajolet, Knuth and Pittel~\cite{FKP89}, Janson \emph{et al.}~\cite{JKLP93} or more recently \cite{EdPR16} for multigraphs.

	\paragraph{Graphs.}
    
A simple graph, or graph, $G$ is a pair $(V(G),E(G))$, where $V(G)$ denotes the set of vertices carrying distinct labels, and $E(G)$ the set of edges.
Each edge is an unoriented pair of distinct vertices, thus loops and multiple edges are forbidden.
An $(n,m)$-graph is a  simple graph with $n$ vertices and $m$ edges.
The labels of the vertices are distinct integers.
When no other constraint is added, the labeling is said to be \emph{general}.
When the vertices are labeled from $1$ to $n$, the labeling is said to be \emph{canonical}.
Unless otherwise mentioned, the graphs considered have canonical labeling.
The set of all simple graphs with canonical labeling is denoted by $\mSG$.

    \paragraph{Multigraphs.}

We define a multigraph as a graph-like object
with labeled vertices, and labeled oriented edges,
where loops and multiple edges are allowed.
More formally, a multigraph $G$ is a pair $(V(G), E(G))$,
where $V(G)$ is the set of labeled vertices
and $E(G)$ the set of labeled edges 
(the edge labels are independent from the vertex labels).
Each edge is a triple $(v,w,e)$,
where $v$, $w$ are vertices,
and $e$ is the label of the edge (which is oriented from $v$ to $w$). 
A loop is a triple $(v,v,e)$ and a multiple edge is a set of at least two edges $\{(u,v,e_1), \dots, (u,v,e_k)\}$.
An $(n,m)$-multigraph is a multigraph with $n$ vertices and $m$ edges.
Again, the vertex labels are distinct integers,
and the edge labels are distinct integers.
When no other constraint is added,
the labeling is said to be \emph{general}.
When the vertices are labeled from $1$ to $n$
and the edges from $1$ to $m$, the labeling is said to be \emph{canonical}.
The set of all multigraphs with canonical labeling is denoted by $\mMG$,
and unless otherwise mentioned, the multigraphs considered have canonical labeling.

Notice that, although a given multigraph may have neither loops nor multiple edges, 
it would still not be a simple graph, as its edges are oriented and labeled. 
The orientation of the edges allows a very simple description of a canonical $(n,m)$-multigraph $G$ as a sequence $(v_1,\dots,v_{2m})$, where $v_i \in V(G)$ and $e_j = (v_{2j-1},v_{2j}) \in E(G), \forall 1\le j \le m$. This model of multigraphs can be found in the literature under the name of \emph{quiver} or \emph{multidigraph}, and appears notably in category theory and representation theory (see~\cite{CMP09}. It introduces a bias in the enumeration of multigraphs and so differs from the more classical model of vertex-labeled, edge-labeled multigraphs since loops have only one possible orientation, while other edges have two.

\medskip

\begin{figure}[h]\label{fig:multi_graphs}
	\centering
	\includegraphics[width=0.6\linewidth]{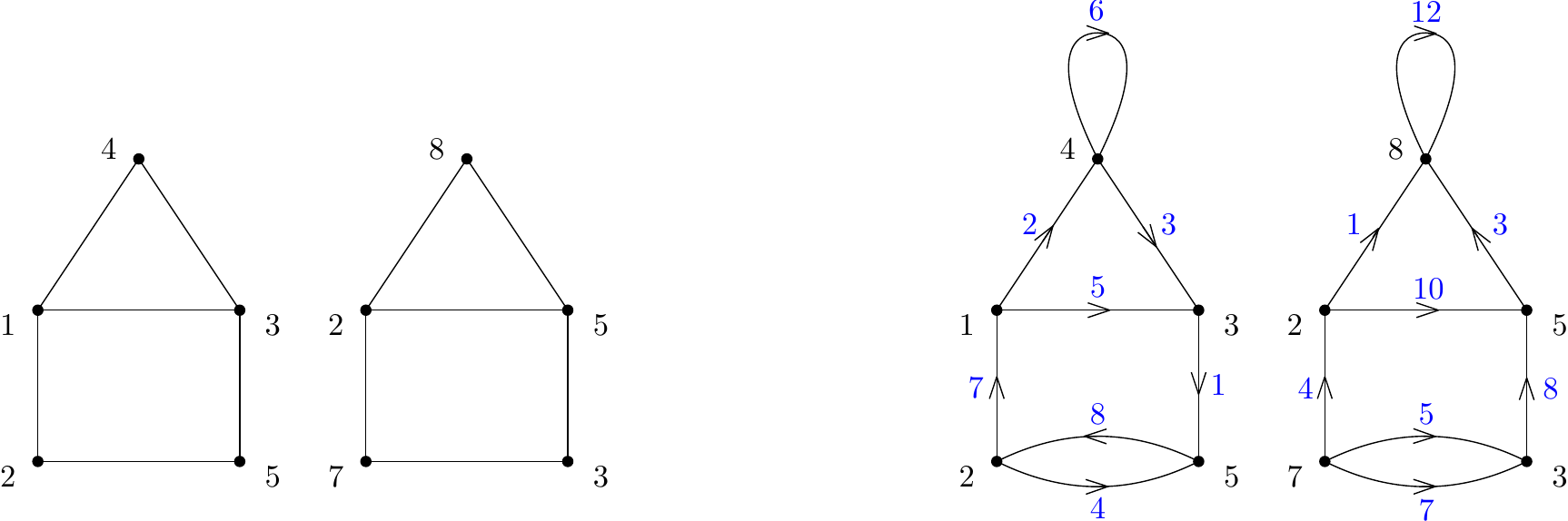}
    \caption{On the left, two isomorphic simple graphs (the first one with canonical labeling). On the right, two isomorphic multigraphs (the first one with canonical labeling).}
\end{figure}

The following definitions stand for both simple graphs and multigraphs. As such, they are stated for (multi)graphs, which can refer either to simple graphs or to multigraphs depending on the context.

	\paragraph{Isomorphic (multi)graphs.}

Two (multi)graphs $G$ and $H$ with general labeling are \emph{isomorphic}
if there exists a bijection $\alpha$ between $V(G)$ and $V(H)$
that induces a bijection $\beta$ between $E(G)$ and $E(H)$, \ie
\[
	\textrm{(for simple graphs) }
    \forall \{v,w\} \in E(G), \beta(\{v, w\}) = \{\alpha(v), \alpha(w)\} \in E(H) ;
\]
\[
	\textrm{(for multigraphs) }
    \forall v,w \in V(G):\ \{\beta(v, w, e) : (v,w,e) \in E(G)\} = \{(\alpha(v), \alpha(w), e') \in E(H)\}.
\]
We also say that $H$ is a $G$-(multi)graph.
Notice that (multi)graph isomorphism is independent from labels. 
Given a (multi)graph family $\mF$, $H$ is a $\mF$-(multi)graph
if it is isomorphic to an element of $\mF$.

	\paragraph{Subgraphs.}

A (multi)graph $F$ is a \emph{subgraph} of a (multi)graph $G$
if $V(F) \subset V(G)$ and $E(F) \subset E(G)$.
We then write $F \subset G$.

Given a (multi)graph family $\mF$ and a (multi)graph $G$,
an $\mF$-subgraph of $G$ is a subgraph of $G$
which is isomorphic to an element of $\mF$.
The number of $\mF$-subgraphs of $G$ is denoted by $G[\mF]$.
When $\mF$ is a singleton $\{H\}$,
we simply write $G[H]$ instead of $G[\{H\}]$.

\begin{figure}[h]\label{fig:subgraph}
	\centering
    \includegraphics[width=0.55\linewidth]{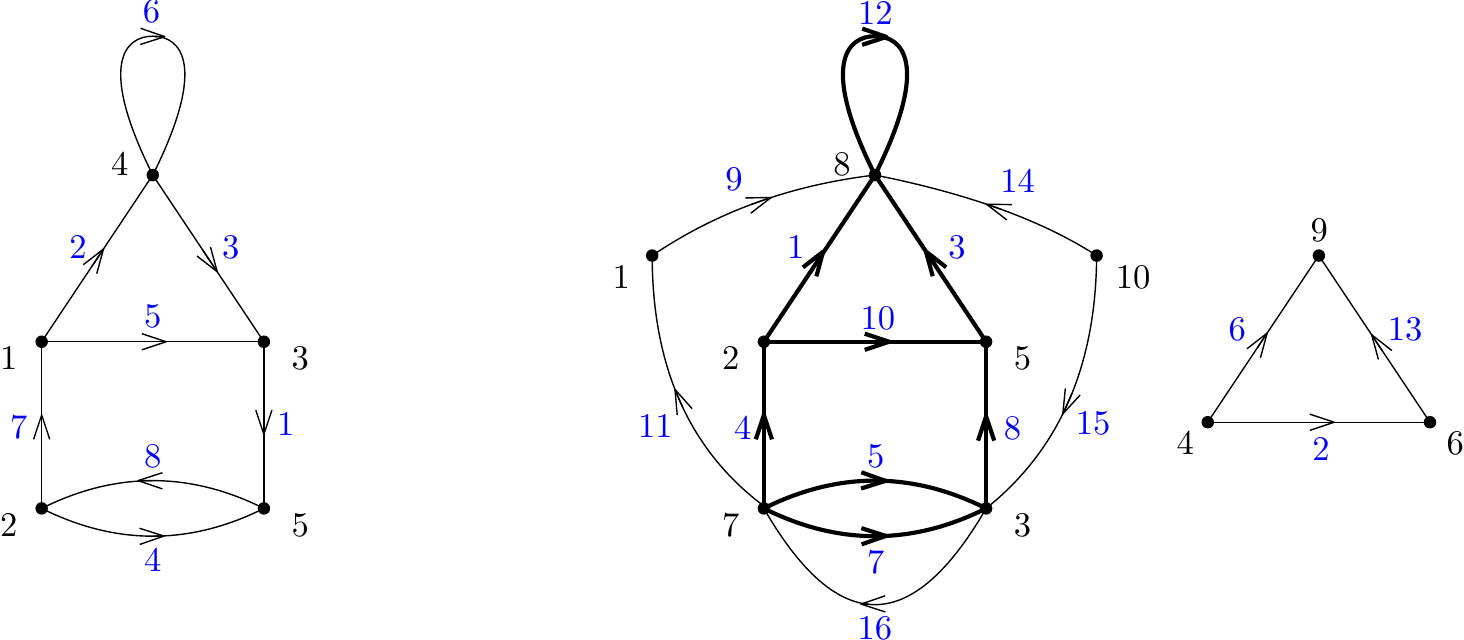}
    \caption{On the left, a multigraph appearing as a subgraph on the (non-connected) multigraph on the right.}
\end{figure}

	\paragraph{Weighted (multi)graphs.}

A \emph{weighted (multi)graph family} $\mF_\omega$
is a (multi)graph family $\mF$ equipped with a \emph{weight} $\omega$,
which is a function from $\mF$ to a given set.
In this article, this set will be either the nonnegative real numbers (see Section~\ref{sec:degree-constraints} for weights related to the degree sequence of the multigraph, or Section~\ref{sec:power-law} for a concrete example with weights following a power law) or the polynomials in the variable $u$ with real coefficients (which can be seen as formal weights, useful to track some parameters in (multi)graphs, see Section~\ref{sec:patchworks}).
The weight of a (multi)graph $G$ is denoted by $\omega(G)$.
The weight $F_{\omega}$ of the family $\mF_{\omega}$ is defined as the sum of the weights of its elements.
The \emph{trivial weight} is the one that assigns to each (multi)graph the value $1$.
When weight $\omega$ is the trivial one, it is omitted in the notations,
so $F$ denotes the trivial weight of the family $\mF$, equal to its cardinality.
Hence, the results on weighted (multi)graphs extend the enumerative results.
When no other weight is specified, the weight is assumed to be the trivial one.

Weights are usually associated to (multi)graph families which are equipped with a non-uniform distribution. For instance, while with the trivial weight each (multi)graph is equally likely, replacing the weight of a single (multi)graph $H$ by $2$ would make it twice as likely to be picked in the new weighted distribution.

We state some general results on the distribution of subgraphs
in a weighted (multi)graph family in Section~\ref{sec:general_weights}.
In Section~\ref{sec:subgraphs}, we obtain more precise results
for the trivial weight, \ie for the usual models of graphs and multigraphs.
Finally, we consider the case where the weight of a (multi)graph depends on the degrees of its vertices in Section~\ref{sec:results-degree-constraints}.

	\paragraph{Counting and probabilities.}

For enumeration purposes, we will only consider families of (multi)graphs with canonical labeling.
Given a family $\mF$ of (multi)graphs,
the set of elements of $\mF$ having $n$ vertices and $m$ edges is denoted by $\mF_{n,m}$,
and its cardinality by $F_{n,m}$.
For instance, $\sg_{n,m}$ is equal to $\binom{\binom{n}{2}}{m}$,
while $\mg_{n,m}$ is equal to $n^{2m}$,
as the labels and orientations of the edges induce
a canonical representation of any multigraph as a sequence of $2m$ vertices.
When a weight $\omega$ is specified,
the weight of the set of all (multi)graphs in the (multi)graph family $\mF_{\omega}$
which have $n$ vertices and $m$ edges is denoted by $F_{n,m,\omega}$.

In our model, a random $(n,m)$-(multi)graph is then
a (multi)graph chosen uniformly at random
from the set $\mSG_{n,m}$ (\resp $\mMG_{n,m}$):
\[
	\proba(G \in \mSG_{n,m}) = \frac{1}{\sg_{n,m}},
    \quad \proba(G\in \mMG_{n,m}) = \frac{1}{\mg_{n,m}}.
\]
This corresponds to the model $\mathcal{G}(n,m)$ introduced by Erd\H{os} and R\'enyi~\cite{ER60}.
When a weight $\omega$ is specified,
then a random $(n,m,\omega)$-(multi)graph is an $(n,m)$-(multi)graph
chosen with probability proportional to its weight:
\[
	\proba(G \in \mSG_{n,m,\omega}) = \frac{\omega(G)}{\sg_{n,m,\omega}},
    \quad \proba(G \in \mMG_{n,m,\omega}) = \frac{\omega(G)}{\mg_{n,m,\omega}}.
\]

	\paragraph{Density and balance.}

Erd\H{os} and R\'enyi~\cite{ER60} observed first
that the analysis of the number of $F$-subgraphs
in a random $(n,m)$-graph is easier
when the graph $F$ is \emph{strictly balanced}.
We recall the definition of this property below.

The \emph{density} of a (multi)graph $G$
is defined as the ratio between the numbers of its edges
and its vertices, and is denoted by
\[
    d(G) = \frac{m(G)}{n(G)}.
\]
By convention, the empty (multi)graph has density $0$.
The \emph{essential density} $d^\star(G)$ of a (multi)graph $G$
is the density of a subgraph of maximal density: 
\[
    d^\star(G) = \max_{H \subset G} d(H).
\]
A (multi)graph $F$ is \emph{strictly balanced} if its density
is greater than the density of all its strict subgraphs:
\[
    d(F) > \max_{H \varsubsetneq F} d(H).
\]
It is \emph{balanced} if
\[
	d(F) \geq \max_{H \varsubsetneq F} d(H),
\]
and \emph{barely balanced} if it is balanced,
but not strictly balanced.
Equivalently a (multi)graph is balanced if and only if $d(F) =  d^\star(F)$.

\begin{lemma} \label{th:density_of_pairs}
Given a connected (multi)graph $F$,
let $\pair_F$ denote the set of (multi)graph
such that $H$ is in $\pair_F$ if and only if
it is obtained by merging two distinct non-disjoint $F$-(multi)graphs $F_1$, $F_2$, \ie
\[
	V(H) = V(F_1) \cup V(F_2), \quad
    E(H) = E(F_1) \cup E(F_2), 
\]
and $V(F_1) \cap V(F_2)$ is non-empty.
If $F$ is strictly balanced, then the density of any (multi)graph from $\pair_F$
is greater than the density of~$F$.
\end{lemma}

	\section{Subgraphs in weighted graphs and multigraphs}
    \label{sec:general_weights}

In this section, we investigate the distribution of finite subgraphs in a random $(n,m,\omega)$-(multi)graph, for a general weight function $\omega$, and reduce the study of this distribution to the analysis of $(n,m,\omega)$-(multi)graphs where an $\mF$-subgraph is distinguished, for a well chosen (multi)graph family $\mF$. The weighted number of all $(n,m,\omega)$-(multi)graphs is denoted by $\mg_{n,m,\omega}$

\smallskip
A (multi)graph $G$ with a distinguished subgraph $F$
can be represented as a pair $(G,F)$.
Consider two weighted families $\mG_{\omega_{\mG}}$ and $\mF_{\omega_{\mF}}$; $\mG_{\omega_{\mG}}^{[\mF_{\omega_{\mF}}]}$ denotes the weighted family of all pairs $(G,F)$
with $G \in \mG$ and $F$ being an $\mF$-subgraph of $G$.
The weight of $(G,F)$ is then implicitly defined as
\[
	\omega((G,F)) := \omega_{\mG}(G) \omega_{\mF}(F).
\]
Therefore, the total weight $\mg^{[\mF_{\omega_{\mF}}]}_{n,m,\omega_{\mG}}$ of all $(n,m)$-(multi)graphs
with a distinguished $\mF_{\omega_{\mF}}$-subgraph is equal to
\[
	\mg_{n,m,\omega_{\mG}}^{[\mF_{\omega_{\mF}}]} :=
    \sum_{\substack{\text{$(n,m)$-(multi)graph $G$}\\ \text{$\mF$-subgraph $F$ of $G$}}} \omega_{\mG}(G) \omega_{\mF}(F).
\]
When no weight function $\omega_{\mF}$ is provided, the trivial weight $1$ is used.
This total weight plays a central role in the article.

\medskip
In the rest of this section, we provide three propositions
that reduce the study of the number $G[\mF]$ of $\mF$-subgraphs in a random $(n,m,\omega)$-(multi)graph $G$
to the analysis of $\mg_{n,m,\omega}^{[\mH_{\omega_{\mH}}]}$,
for well chosen families $\mH_{\omega_{\mH}}$.
The expected value of $G[\mF]$ is computed in Proposition~\ref{th:general_weight_mean} using 
$\mg_{n,m,\omega}^{[\mH_{\omega_{\mH}}]}$.
Proposition~\ref{th:number_of_subgraphs_general_weights} then gives an exact expression for the total weight of $(n,m,\omega)$-(multi)graphs with exactly $t$ $\mF$-subgraphs.
Finally, Proposition~\ref{th:general_weights_second_asymptotics}
provides technical conditions on the weights
so that the limit law of $G[F]$ is a Poisson law for any strictly balanced (multi)graph $F$.

	\subsection{Expected number of subgraphs}

\begin{proposition}[Expected number of subgraphs, weights, simple graphs and multigraphs] 
\label{th:general_weight_mean}
The expected number of $\mF$-subgraphs
in a random $(n,m,\omega)$-graph and a random $(n,m,\omega)$-multigraph is
\[
	\frac{\sg_{n,m,\omega}^{[\mF]}}{\sg_{n,m,\omega}} 
    \quad \text{and} \quad
	\frac{\mg_{n,m,\omega}^{[\mF]}}{\mg_{n,m,\omega}}, \quad\text{respectively}, 
\]
where $\sg_{n,m,\omega}$ is the weighted number of all $(n,m,\omega)$-graphs,
and $\sg_{n,m,\omega}^{[\mF]}$ is the weighted number of all $(n,m,\omega)$-graphs
with a distinguished $\mF$-subgraph (and similarly for multigraphs).
\end{proposition}

\begin{proof}
We give the proof for graphs, the proof for multigraphs being identical.
By definition, the expected number of $\mF$-subgraphs in a random $(n,m,\omega)$-graph
is equal to
\[
	\frac{1}{\sg_{n,m,\omega}}
    \sum_{\substack{G \in \mSG_{n,m}\\ F \in \mF}}
    G[F] \omega(G).
\]
Since $G$ contains $G[F]$ $F$-subgraphs, there are $G[F]$ pairs $(G,F)$ in $\mSG_{n,m}^{[\mF]}$,
so the expected number is also equal to
\[
	\frac{1}{\sg_{n,m,\omega}}
    \sum_{(G,F) \in \mSG_{n,m,\omega}^{[\mF]}} \omega(G) =
    \frac{\sg_{n,m,\omega}^{[\mF]}}{\sg_{n,m,\omega}}.\qedhere 
\]
\end{proof}

    	\subsection{Exact number of subgraphs and patchworks} \label{sec:patchworkdef}

When counting the number of (multi)graphs with $n$ vertices, $m$ edges and containing exactly $t$ $\mF$-subgraphs,
we run into a difficulty: those subgraphs may overlap.
To describe these overlaps, we use the notion of an $\mF$-\emph{patchwork}, defined below.

	\paragraph{Patchworks for multigraphs.}
    
Given a family $\mF$ of multigraphs,
an $\mF$-patchwork is a finite set of distinct $\mF$-multigraphs (called \emph{pieces})
\[
    P = \{ (V_1, E_1), \ldots, (V_{|P|}, E_{|P|})\},
\]
each carrying a general labeling (\ie the vertex or edge labels need not be consecutive integers staring at $1$)
which may share vertices and edges, such that
if two vertices (from two distinct pieces) share the same labels,
then they are merged and similarly,
if two edges from distinct pieces share the same label, then they are merged.
In particular, this implies that two edges sharing a label
must connect the same two vertices with the same orientation.
The vertices and edges of $P$ are
\begin{equation} 
\label{PW-graph}
    V(P) = \bigcup_{i=1}^{|P|} V_i, \quad \text{and} \quad E(P) = \bigcup_{i=1}^{|P|} E_i,
\end{equation}
and their cardinalities are denoted by $n(P)$ and $m(P)$.
The size of $P$ is the number of pieces and is denoted by $|P|$.
This notion is illustrated in Figure~\ref{fig:patchwork_multi}.

\begin{figure}[h]
\begin{center}
\includegraphics[width=\linewidth]{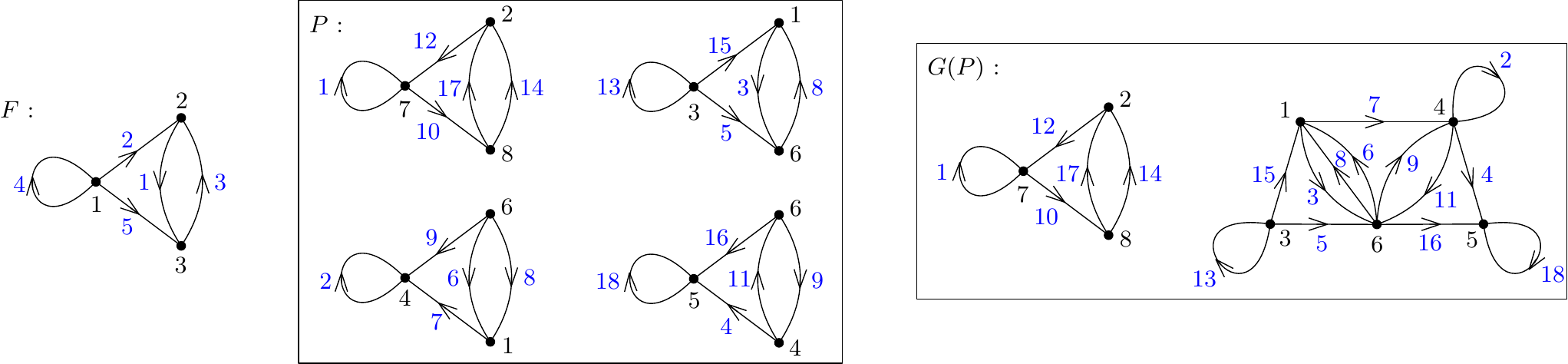}
\caption{A multigraph $F$ (left) and a $F$-patchwork $P$ of size 4 (center) with its representation $G(P)$ as a multigraph (right).}
\label{fig:patchwork_multi}
\end{center}
\end{figure}

A patchwork is called \emph{disjointed} if it contains no pair of pieces sharing one vertex or more.

	\paragraph{Patchworks for simple graphs.}

The notion of a patchwork can be naturally adapted to simple graphs.
For a simple graph family $\mF$, each piece of an $\mF$-patchwork
is then an $\mF$-subgraph and pieces can share vertices and edges.
As distinguished from multigraphs, edges connecting the same two vertices in distinct pieces
are necessarily merged in the patchwork, so as to avoid multiple edges.

\begin{figure}[h]
\begin{center}
\includegraphics[scale=0.8]{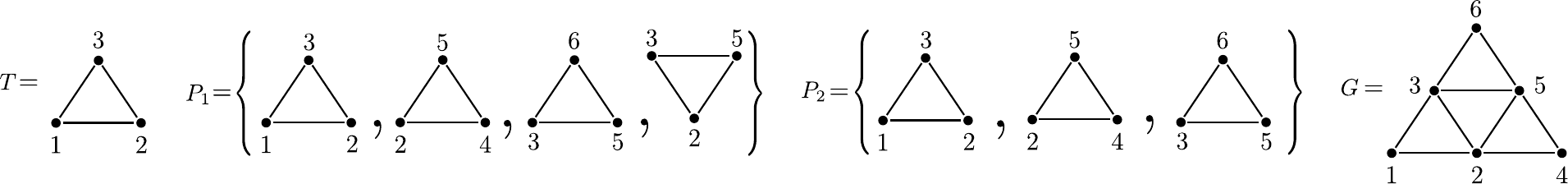}
\caption{A graph $T$ and two $T$-patchworks $P$ and $P'$ that have the same graph representation $G(P) = G(P') = G$.} 
\label{fig:patchwork}
\end{center}
\end{figure}

	\paragraph{Weights of patchworks and exact number of subgraphs.}

The family $\mPatch{\mF}(u)$ denotes the set of $\mF$-patchworks
equipped with the weight function
\[
	\omega(P) := u^{|P|},
\]
so the weight of a patchwork $P$ is the variable $u$
raised to the power equaling the number of pieces of $P$.

Similarly, the family $\disjointpatch{\mF}$ denotes the subset of disjointed $\mF$-patchworks.

\begin{proposition}[Total weight, simple and multigraphs]
\label{th:number_of_subgraphs_general_weights}
The total weight $\sg_{n,m,\omega,t}^{\mF}$ of all $(n,m,\omega)$-graphs
that contain exactly $t$ $\mF$-subgraphs is 
\[
	\sg_{n,m,\omega,t}^{\mF} := [u^t] \sg_{n,m,\omega}^{[\mPatch{\mF}(u-1)]}.
\]
Likewise, for any multigraph family $\mF$,
the total weight $\mg_{n,m,\omega,t}^{\mF}$ of all $(n,m,\omega)$-multigraphs
that contain exactly $t$ $\mF$-subgraphs is
\[
	\mg_{n,m,\omega,t}^{\mF} := [u^t] \mg_{n,m,\omega}^{[\mPatch{\mF}(u-1)]}.
\]
\end{proposition}

\begin{proof}
We give the proof for graphs, as the proof for multigraphs is identical.
It relies on the interpretation of the inclusion-exclusion principle
commonly used in analytic combinatorics (see Flajolet and Sedgewick~\cite[Section III.7.4]{FS09}).
Recall that $G[\mF]$ denotes the number of occurrences
of $\mF$-subgraphs in the graph $G$.
Let $\mSG_{n,m,\omega}^{\mF}(u)$ denote the graph family
that contains all $(n,m)$-graphs,
equipped with the weight $\omega_u$ defined as
\[
	\omega_u(G) := \omega(G) u^{G[\mF]}.
\]
Its weight is denoted by $\sg_{n,m,\omega}^{\mF}(u)$ and is equal to
\[
	\sg_{n,m,\omega}^{\mF}(u) =
    \sum_{G \in \mSG_{n,m}} \omega(G) u^{G[\mF]},
\]
and the weight of the $(n,m,\omega)$-graphs that contain exactly $t$ $\mF$-subgraphs is
then the $t$-th coefficient of $\sg_{n,m,\omega}^{\mF}(u)$
\begin{equation} \label{eq:number_of_subgraphs_general_weights}
	\sg_{n,m,\omega,t}^{\mF} = [u^t] \sg_{n,m,\omega}^{\mF}(u).
\end{equation}
When we evaluate $\sg_{n,m,\omega}^{\mF}(u)$ at $u+1$ instead of $u$
and develop the powers of $u+1$, we obtain
\[
	\sg_{n,m,\omega}^{\mF}(u+1) =
    \sum_{G \in \mSG_{n,m}} \sum_{j=0}^{G[\mF]} \omega(G) \binom{G[\mF]}{j} u^j.
\]
The binomial coefficient $\binom{G[\mF]}{j}$ is equal
to the number of $\mF$-patchworks with $j$ pieces contained in $G$,
and $u^j$ is the weight of each of those patchworks.
Hence, $\sg_{n,m,\omega}^{\mF}(u+1)$ is also the total weight
of the $(n,m,\omega)$-graphs where an $\mF$-patchwork is distinguished:
\[
	\sg_{n,m,\omega}^{\mF}(u+1) =
    \sum_{(G,P) \in \mSG_{n,m}^{[\mPatch{\mF}(u)]}} \omega(G) u^{|P|} =
    \sg_{n,m,\omega}^{[\Patch{\mF}(u)]}.
\]
Replacing $u$ with $u-1$, extracting the coefficient $[u^t]$,
and injecting the result in Equation~\eqref{eq:number_of_subgraphs_general_weights}
finishes the proof.
\end{proof}

	\subsection{Limit law of the number of subgraphs}

Let $c$ denote a real value or a polynomial in the variable $u$, and $F$ a connected (multi)graph 
(in this section $c$ will denote the weight of~$F$).
For simplification we will deviate from standard notations and use 
$$e^{c F}:=\mathtt{Set}(c \{F\})$$ 
to denote the set of (multi)graphs where each component is isomorphic to $F$.
The weight of such a (multi)graph is then defined as $c^k$,
where $k$ is the number of components.
Similarly, given a (multi)graph $H$,
we write  
\begin{equation}
\label{eq:def_exp}
H e^{c F}:=\{H\}\ast\mathtt{Set}(c\{F\})
\end{equation}
to denote the set of (multi)graphs
where one distinguished connected component is isomorphic to $H$,
while all the others are isomorphic to $F$.
The weight of a (multi)graph from this family is defined
as $c^k$, where $k$ is the number of components isomorphic to $F$.
By convention, when $H$ is the empty (multi)graph (\ie the graph that has no vertex and no edge),
then $H e^{c F}$ is equal to $e^{c F}$.

\begin{lemma} \label{th:general_weights_first_asymptotics}
Let $F$ be a connected multigraph, and $\pair_F$ denote the (finite) set of multigraphs
obtained by merging two distinct isomorphic copies of $F$
sharing at least one vertex.
Consider a sequence $m := m(n)$ and recall that $\mg_{n,m,\omega}^{[\mG]}$ denotes the cumulative weight of all $(n,m,\omega)$-multigraphs with a distinguished $\mG$-subgraph. 
Assume that the following asymptotic relations hold, as $n$ tends to infinity: 
\begin{itemize}
\item (Case $c=2$)
For all $H\in\pair_F$, we have $\mg_{n,m,\omega}^{[H e^{2 F}]}=\smallo(\mg_{n,m,\omega}^{[F]})$. 
\item (Case $c=u-1$)
For $t$ fixed, $[u^t] \mg_{n,m,\omega}^{[e^{(u-1) F}]}=\exactbigO(\mg_{n,m,\omega}^{[F]})$.
\end{itemize}
Then the weight of all $(n,m,\omega)$-multigraphs
that contain exactly $t$ $F$-subgraphs satisfies
\[
	\mg_{n,m,\omega,t}^F \sim 
    [u^t] \mg_{n,m,\omega}^{[e^{(u-1) F}]}.
\]
\end{lemma}

\begin{proof}
Let $\disjointmMG^{F}$ denote the set of multigraphs $G$
such that the $F$-subgraphs of $G$ are disjoint (\ie share no vertex),
and $\jointmMG^{F}$ the complementary set.
As usual, let $\disjointmg_{n,m,\omega,t}^F$ (\resp $\jointmg_{n,m,\omega,t}^F$)
denote the weight of the $(n,m,\omega)$-multigraphs from $\disjointmMG^F$ (\resp $\jointmMG^F$)
with exactly $t$ $F$-subgraphs.
Then their sum is equal to the weight of $\mMG_{n,m,\omega,t}^F$
\[
    \mg_{n,m,\omega,t}^F =
    \disjointmg_{n,m,\omega,t}^F +
    \jointmg_{n,m,\omega,t}^F.
\]
Applying the same inclusion-exclusion principle
as in the proof of Theorem~\ref{th:number_of_subgraphs_general_weights},
we obtain
\[
    \disjointmg^{F}_{n,m,\omega,t} =
    [u^t] \disjointmg_{n,m,\omega}^{[\disjointpatch{F}(u-1)]}.
\]
The right-hand side can be decomposed using the relation
\[
    \mg_{n,m,\omega}^{[\disjointpatch{F}(u-1)]} =
    \disjointmg_{n,m,\omega}^{[\disjointpatch{F}(u-1)]} +
    \jointmg_{n,m,\omega}^{[\disjointpatch{F}(u-1)]}.
\]
Combining the last three equations, we obtain
\begin{equation} \label{eq:strict_balanced_mgnmt}
    \mg_{n,m,\omega,t}^F =
    [u^t] \mg_{n,m,\omega}^{[\disjointpatch{F}(u-1)]} -
    [u^t] \jointmg_{n,m,\omega}^{[\disjointpatch{F}(u-1)]} +
    \jointmg_{n,m,\omega,t}^F.
\end{equation}
We now consider the right-hand side,
express the first term,
and prove that the third term
and the absolute value of second term are negligible.
\proofparagraph{First term.}
Recall that a disjointed $F$-patchwork is a set of mutually disjoint multigraphs which are isomorphic to $F$, so
\[
	[u^t] \mg_{n,m,\omega}^{[\disjointpatch{F}(u-1)]} = [u^t] \mg_{n,m,\omega}^{[e^{(u-1) F}]},
\]
which is of order  $\exactbigO(\mg^{[F]}_{n,m,\omega})$.
\proofparagraph{Third term.}
The $(n,m)$-multigraphs with $t$ $F$-subgraphs form a subset of all $(n,m)$-multigraphs, so
\[
	\jointmg_{n,m,\omega,t}^F \leq \jointmg_{n,m,\omega}.
\]
Any multigraph with a distinguished $\pair_F$-subgraph
is in $\jointmMG^F$, so
\[
	\jointmg_{n,m,\omega} \leq \mg_{n,m,\omega}^{[\pair_F]} = \sum_{H \in \pair_F} \mg_{n,m,\omega}^{[H]}.
\]
Since $\mg_{n,m,\omega}^{[H e^{2 F}]}=\smallo(\mg^{[F]}_{n,m,\omega})$, so is $\mg_{n,m,\omega}^{[H]}$, and so
the third term of the right-hand side of Equation~\eqref{eq:strict_balanced_mgnmt} is negligible compared to the first one.
\proofparagraph{Second term.}
The absolute value of the second term is bounded by
\begin{align*}
    \left|
    [u^t] \jointmg_{n,m,\omega}^{[\disjointpatch{F}(u-1)]}
    \right|
    &\leq
    [u^t] \jointmg_{n,m,\omega}^{[e^{(u+1)F}]} \\
    &\leq
    \sum_{t\ge 0} [u^t] \jointmg_{n,m,\omega}^{[e^{(u+1)F}]} 1^t 
    =
    \jointmg_{n,m,\omega}^{[e^{2 F}]},
\end{align*}
which is the weight of the set of all multigraphs from $\jointmMG_{n,m}$
where a disjointed patchwork $P$ is distinguished,
and counted with a weight $2^{|P|}$.
Let $\mK$ be the finite set of graphs composed of a multigraph $H$ from $\pair_F$
and a (possibly empty) set of disjoint copies of $F$, each sharing at least one vertex with $H$. $\pair_F$ being finite, and each element having at most $2 V(F) -1$ vertices, the set $\mK$ is also finite.
There exists at least one element of $\mK$ in each graph from $\jointmMG$, thus the following holds:
\[
	\jointmg_{n,m,\omega}^{[e^{2 F}]} \leq \sum_{K \in \mK} \mg_{n,m,\omega}^{[K e^{2 F}]}.
\]
Since the sum is finite and each terms is of order $\smallo(\mg^{[F]}_{n,m,\omega})$,
the second term of the right-hand side of Equation~\eqref{eq:strict_balanced_mgnmt} is also negligible,
which concludes the proof.
\end{proof}

The following proposition provides simple conditions
on the asymptotic weight of multigraphs with a distinguished subgraph
to conclude that the number of $F$-subgraphs has a Poisson limit 
when $F$ is strictly balanced.

\begin{proposition}[Poisson distribution, strictly balanced, multigraphs] 
\label{th:general_weights_second_asymptotics}
Let $m := m(n)$ denote an integer sequence going to infinity with $n$,
$F$ be a strictly balanced multigraph,
and $\omega$ a weight function.
Let us assume that there exists a sequence of functions $\lambda_n:\mMG\to\mathbb R$
satisfying the following assertions: 
\begin{itemize}
\item
The limit $\lambda(F) := \lim_{n\to\infty}\lambda_n(F)$ exists and is positive.
\item
For any constant $c$ and any multigraph $H$, we have
\[
	\mg_{n,m,\omega}^{[H e^{c F}]} \sim 
    \lambda_n(H) e^{c \lambda_n(F)} \mg_{n,m,\omega}, 
\]
where we used the notation introduced in~\eqref{eq:def_exp}.
\item
If $d(G)>d(F)$ then $\lambda(G)=0$.
\end{itemize}
Then the number of $F$-subgraphs in a random $(n,m,\omega)$-multigraph
follows a Poisson limit law with parameter $\lambda(F)$.
\end{proposition}

\begin{proof}
Since $F$ is connected and strictly balanced,
according to Lemma~\ref{th:density_of_pairs},
any multigraph $H$ in $\pair_F$ is denser than $F$.
The third hypothesis of the proposition implies
that $\lambda_n(H)$ tends to $0$, so
\[
	\mg_{n,m,\omega}^{[H e^{2 F}]} \sim 
    \lambda_n(H) e^{2 \lambda_n(F)} \mg_{n,m,\omega} =
    \smallo(\mg_{n,m,\omega}).
\]
Thus, the first hypothesis of Lemma~\ref{th:general_weights_first_asymptotics} is satisfied.
In the last part of the proof, we will prove
\begin{equation} \label{eq:second_part_general_weights_second_asymptotics}
	[u^t] \mg_{n,m,\omega}^{[e^{(u-1) F}]} \sim 
    \frac{\lambda_n(F)^t}{t!}
    e^{-\lambda_n(F)}
    \mg_{n,m,\omega}.
\end{equation}
This implies
\[
	[u^t] \mg_{n,m,\omega}^{[e^{(u-1) F}]} = \exactbigO(\mg_{n,m,\omega}),
\]
so the second hypothesis of Lemma~\ref{th:general_weights_first_asymptotics} is satisfied.
Applying this lemma, we then obtain
\[
	\mg_{n,m,\omega,t}^F \sim 
    [u^t] \mg_{n,m,\omega}^{[e^{(u-1) F}]} \sim
    \frac{\lambda_n(F)^t}{t!}
    e^{-\lambda_n(F)}
    \mg_{n,m,\omega},
\]
so the number of $F$-subgraphs in a random $(n,m,\omega)$-multigraph
has a Poisson limit law of parameter~$\lambda(F)$.
\proofparagraph{Proof of Equation~\eqref{eq:second_part_general_weights_second_asymptotics}.}
We follow the proof of \cite[Theorem IX.1]{FS09}.
For any $u$ in the unit disk, the function sequence
\[
	f_n(u) = \frac{\mg_{n,m,\omega}^{[e^{(u-1) F}]}}{\mg_{n,m,\omega}}
\]
tends to $e^{(u-1) \, \lambda(F)}$,
and
\[
	|f_n(u)| \leq
    \frac{\mg_{n,m,\omega}^{[e^{(|u-1|) F}]}}{\mg_{n,m,\omega}} \leq
    \frac{\mg_{n,m,\omega}^{[e^{2 F}]}}{\mg_{n,m,\omega}}.
\]
Since the right-hand side has a finite limit, it is bounded,
so $|f_n(u)|$ is uniformly bounded on the unit disk.
Then, according to Vitali's Theorem, the sequence of functions $f_n(u)$
converges uniformly in a neighborhood of $0$ to $e^{(u-1) \, \lambda(F)}$,
which implies Equation~\eqref{eq:second_part_general_weights_second_asymptotics}.
\end{proof}


		\section{Generating functions for graphs and multigraphs}
		\label{sec:fgs-graphs-multigraphs}
    

		\subsection{Analytic combinatorics}

We recall here briefly the technique of translating combinatorial operations into equations for generating functions.

    \paragraph{Symbolic Method.}

The book of Flajolet and Sedgewick \cite{FS09} provides an excellent introduction to the techniques of analytic combinatorics.
The main idea is to associate to any combinatorial family $\mA$ of canonically labeled objects a generating function
\[
    A(z) = \sum_{n \geq 0} a_n \frac{z^n}{n!},
\]
where $a_n$ denotes the number of objects of size $n$ in $\mA$.
In the present article, to express the generating function of interesting combinatorial families, we apply the following dictionary to translate combinatorial relations between the families into analytic equations on their generating functions.
\begin{itemize}
\item
\textbf{Disjoint union.}
If $\mA \cap \mB = \emptyset$, and $\mC = \mA \cup \mB$,
then $C(z) = A(z) + B(z)$.
\item
\textbf{Relabeled Cartesian product.}
In the relabeled Cartesian product $\mC = \mA \times \mB$,
we consider all relabelings of the pairs $(a,b)$, with $a \in \mA$ and $b \in \mB$,
so that each label, from $1$ to the sum of the sizes of $a$ and $b$, appears exactly once, \emph{i.e.}, $(a,b)$ is canonically labelled.
We then have
\[
    C(z) = A(z) B(z).
\]
\item
\textbf{Set.}
The family of sets of objects from $\mA$ (where the elements of a set are relabeled such that the set is canonically labeled after all) has generating function $e^{A(z)}$.
\end{itemize}

		\subsection{Multigraphs.}
        \label{sec:def-multigraphs}

Since vertices and edges are labeled, 
we choose exponential generating functions
with respect to both quantities
(see Bergeron, Labelle and Leroux~\cite{BLL97} or again~\cite{FS09}).
Furthermore, a weight $1/2$ is assigned to each edge to take into account the orientation.
The generating function of a multigraph family $\mF$ is then
\[
  F(z,w) =
  \sum_{G \in \mF}
  \frac{w^{m(G)}}{2^{m(G)} m(G)!}
  \frac{z^{n(G)}}{n(G)!},
\]
and the number of multigraphs with $n$ vertices and $m$ edges
in the family $\mF$ is
\[
    F_{n,m} = n! 2^m m! [z^n w^m] F(z,w).
\]
For example, the generating function
of the set $\mMG$ of all multigraphs is
\[
    \mg(z,w) = \sum_{n \geq 0} e^{w n^2/2} \frac{z^n}{n!}.
\]

    \paragraph{Weighted multigraphs.}

A \emph{weighted} multigraph family is a family $\mF_{\omega}$
where each multigraph $G$ comes with a weight $\omega(G)$.
For example, a multigraph family is a particular case of a weighted family,
where each multigraph has weight $1$.
The weight of the weighted family $\mF$ is
\[
    F_{\omega} = \sum_{G \in \mF} \omega(G),
\]
and its generating function is
\[
    F_{\omega}(z,w) =
    \sum_{G \in \mF}
    \omega(G)
    \frac{w^{m(G)}}{2^{m(G)} m(G)!}
    \frac{z^{n(G)}}{n(G)!}.
\]
The weight of $\mF_{n,m,\omega}$, the set of multigraphs from $\mF_{\omega}$ with $n$ vertices and $m$ edges, is then
\[
    F_{n,m,\omega} = n! 2^m m! [z^n w^m] F_{\omega}(z,w).
\]

    \paragraph{Subgraphs.}

In this article, we consider a uniform random multigraph $G$ in $\mMG_{n,m}$, where $n$ and $m$ are large integers, and aim at deriving the limit law for the number $G[\mF]$  of $\mF$-subgraphs for any finite family $\mF$.
The generating function of all multigraphs where the number of $\mF$-subgraphs is marked (by the variable~$u$) is denoted by
\[
    \mg^\mF(z,w,u) =
    \sum_{G \in \mMG}
    u^{G[\mF]}
    \frac{w^{m(G)}}{2^{m(G)} m(G)!}
    \frac{z^{n(G)}}{n(G)!}.
\]

Therefore, the number of multigraphs
with $n$ vertices, $m$ edges,
and that contain exactly $t$ $\mF$-multigraphs is
\[
    \mg_{n,m,t}^\mF =
    n! 2^m m! [z^n w^m u^t] \mg^\mF(z,w,u).
\]
Notice that this formula corresponds to weighted multigraphs where each graph $G\in\mMG$ has weight $\omega(G) = u^{G[\mF]}$.

\paragraph{Multivariate vs. weighted series.} The generating series $\mg^\mF(z,w,u)$, in three variables $z$, $w$ and $u$, can also be seen as a bivariate series in $z$ and $w$, for the weighted multigraph family where a graph $G\in \mMG$ has weight $\omega(G)= u^{G[\mF]}$.

In the following, we will sometimes exploit this equivalence between multivariate generating series and weighted generating series.

        \subsection{From multigraphs to graphs} 

Most of our results are first derived for multigraphs, because this model is better suited for generating function manipulations.
However, the most common model in the graph and combinatorics communities is the simple graph model.

In a multigraph, a \emph{loop} is an edge linking a vertex to itself.
A \emph{multiple edge} is a set of edges between the same two distinct vertices.
A \emph{graph} is a multigraph where the edges are neither labeled nor oriented and loops as well as multiple edges are forbidden.
Thus, a simple graph on $n$ vertices contains at most $\binom{n}{2}$ edges.
Since the edges are non-labeled and non-oriented, we define the generating function of a simple graph family $\mF$ as
\[
    F(z,w) =
    \sum_{G \in \mF}
    w^{m(G)}
    \frac{z^{n(G)}}{n(G)!}.
\]
For example, the generating function of all simple graphs is
\[
  \sg(z,w) = 
  \sum_{n \geq 0}
  (1+w)^{\binom{n}{2}}
  \frac{z^n}{n!}.
\]

\begin{lemma} \label{th:multigraphs_to_graphs}
Consider a multigraph family $\mF$ which is 
stable with respect to edge relabeling and change of orientation
and containing neither loops nor multiple edges.
We build the simple graph family $\mH$ from $\mF$
by removing the edge labels and orientations.
Then, the generating functions of $\mF$ and $\mH$ are equal:
\[
    F(z,w)
    =
    \sum_{G \in \mF}
    \frac{w^{m(G)}}{2^{m(G)} m(G)!}
    \frac{z^{n(G)}}{n(G)!}
    =
    \sum_{G \in \mH}
    w^{m(G)}
    \frac{z^{n(G)}}{n(G)!}
    =
    H(z,w).
\]
\end{lemma}

\begin{proof}
There exist $2^{m(G)} m(G)!$ possible orientations and labelings for the edges of any simple graph $G$.
The multigraphs obtained contain neither loops nor multiple edges.
Conversely, any multigraph family being stable with respect to edge relabeling and change of orientation and containing 
neither loops nor multiple edges corresponds to a unique graph family.
\end{proof}

Let $\mF$ denote a simple graph family, and $\mH$ the corresponding multigraph family, obtained by labeling and orienting the edges in all possible ways.
In light of the previous lemma, counting graphs from $\mSG_{n,m}$ with $t$ copies from $\mF$
is equivalent to counting multigraphs from $\mMG_{n,m}$ with no loops and no multiple edges and containing $t$ copies from $\mH$.
Removing the loops and multiple edges can be realized by forbidding subgraphs from the multigraph family that consists of the following two elements: the multigraph being a single loop and the multigraph which forms a double edge. There are actually four multigraphs comprising solely a double edge (by relabeling and re-orienting the edges), but as the notion of subgraph is independent of the labeling, only one is needed.

 		\subsection{Patchworks}
        \label{sec:patchworks}

A patchwork $P$ gives rise to a (multi)graph $G(P):=(V(P), E(P))$ (\emph{cf.} \eqref{PW-graph}). However, as seen in Figure~\ref{fig:patchwork},
this underlying (multi)graph is not enough to characterize a patchwork.
Some information is lost, e.g., the number of pieces in the patchwork which turns out to be useful for enumeration purposes.

As was already the case for (multi)graphs, a patchwork $P$ has a canonical labeling if its (multi)graph representation $G(P)$ has a canonical labeling. The set of all $\mF$-patchworks with canonical labeling is denoted by $\Patch\mF$, and its generating function is
\begin{align*}
	\Patch{\mF}(z,w,u) &=
    \sum_{P \in \Patch{\mF}}
    u^{|P|}
    \frac{w^{m(P)}}{2^{m(P)} m(P)!}
    \frac{z^{n(P)}}{n(P)!} \qquad \textrm{(for multigraphs)}; \\
    \Patch{\mF}(z,w,u) &=
    \sum_{P \in \Patch{\mF}}
    u^{|P|}
    w^{m(P)}
    \frac{z^{n(P)}}{n(P)!} \qquad \textrm{(for simple graphs)}.
\end{align*}

\begin{proposition}[Generating function, patchworks] 
\label{th:patchwork}
Given a (multi)graph family $\mF$, closed by isomorphism and with generating function $F(z,w)$,
the generating function of disjointed patchworks is
\[
    \disjointpatch{\mF}(z,w,u) = e^{u F(z,w)}.
\]
The generating function of all patchworks is
\[
    \Patch{\mF}(z,w,u) = e^{u F(z,w)} (1 + Q^\mF(z,w,u)),
\]
where $Q^\mF(z,w,u)$ denotes the generating function of patchworks that contain no isolated piece,
\ie each piece shares at least one vertex with another piece.
\end{proposition}


        \section{Number of small subgraphs in a random (multi)graph}
        \label{sec:subgraphs}


We consider here simple graphs and multigraphs without degree constraints and examine what can be said about the number of occurrences of subgraphs in the whole graph. 
First, we derive the exact and asymptotic numbers of (multi)graphs with a distinguished subgraph belonging to a family~$\mF$ in Sections~\ref{sec:distinguished_multi} and~\ref{sec:distinguished_simple}.

Next, we obtain an exact formula for the number of subgraphs of $\mF$ in Section~\ref{sec:enum-given-number-subgraphs}.
Finally, we apply our results to a variety of problems, some old and some new ones, in Section~\ref{sec:applications-to-subgraphs}: probability of occurrence of a ``high-density'' subgraph in a multigraph, densest subgraph in a simple graph, limiting Poisson distributions for the number of occurrences of a graph in a simple graph and of a strictly balanced graph in a multigraph.

        \subsection{Multigraphs with one distinguished subgraph}
        \label{sec:distinguished_multi}

We first consider here how we can obtain the generating function of the class of all (multi)graphs with one distinguished subgraph, which belongs to the family $\mF$.

\medskip
Given two multigraph families $\mH$ and $\mF$, we denote by $\mH^{[\mF]}$ the set of multigraphs from $\mH$ where exactly one copy of a multigraph from $\mF$ is distinguished.
If $\mF$ is a weighted family, then the weight of $G \in \mH^{[\mF]}$
is equal to the weight of the distinguished copy from $\mF$.

\begin{theorem}[Distinguished, exact/asymptotics, multigraphs] 
\label{th:distinguished_multi}

i) 
The number (or total weight) of multigraphs with $n$ vertices, $m$ edges, and that contain one distinguished subgraph from $\mF$, is equal to
\begin{equation}
	\mg^{[\mF]}_{n,m} =
    n! 2^m m! [z^n w^m] F(z,w) e^z e^{n^2 w/2}.
\label{eq:distinguished_multi_exact}
\end{equation}

ii) 
Assume that the generating function $F$ of the family $\mF$ is entire in both variables.
Let $D$ denote an open set containing the compact unit disk.
Let $n$ and $m = m(n)$ be two integers going to infinity such that $\frac{F(n z, 2m w / n^2)}{F(n, 2m / n^2)}$ converges uniformly on each compact $K \in D^2$ to an analytic function on $D^2$, $L(z,w)$.
Then the set $\mMG^{[\mF]}_{n,m}$ of multigraphs containing $n$ vertices, $m$ edges, and with one copy from~$\mF$ distinguished, has weight
\begin{equation}
    \mg^{[\mF]}_{n,m} \sim
    n^{2m} F \left( n, \frac{2m}{n^2} \right).
\label{eq:distinguished_multi_asympt}
\end{equation}
\end{theorem}

\paragraph{Remark.} Although assuming that the function $F$ is entire in both variables might seem restrictive, it is in fact quite natural: if the family $\mF$ is finite, then the function $F(z,w)$ is a polynomial.
An example of an infinite family with infinite radius of convergence in both variables would be the family of cycles with increasing labelling on edges and vertices (the first vertex is labelled by~1, then the edge between 1 and~2 is labelled by~1, and so on).
But the family $\mF$ of all labelled cycles does not satisfy the assumption of the theorem, as it leads to a generating function satisfying $F(z,1)=\log \frac{1}{1-z}$, which has radius of convergence~1.

In Section~\ref{sec:applications-to-subgraphs}, Theorem~\ref{th:distinguished_multi} will also be applied when $\mF$ is the set of elements from a finite family, whose generating function, the exponential of a polynomial, is indeed entire.

\begin{proof}
A multigraph on $n$ vertices, where one $\mF$-subgraph is distinguished, is a copy of a multigraph~$F$ from $\mF$, a set of additional vertices, and a set of additional edges.
The symbolic method (see \cite{FS09}) translates this combinatorial description into an expression for the generating function, from which we extract the desired coefficient to get the exact expression of~\emph{i)}.
The asymptotics of \emph{ii)}\ is then extracted using a saddle-point method.
We now consider each of these points in detail.
 
    \paragraph{Exact expression.}
To build a multigraph, we start with a distinguished copy of a subgraph from $\mF$ and add first a set of vertices.
So far, this family is described by the generating function
\[
	F(z,w) e^z.
\]
Now, we add edges.
If the multigraph we want to obtain at the end contains $n$ vertices, then the number of possible edges is $n^2$.
With our convention (an edge has weight $\frac{1}{2}$), the generating function of one edge is $w/2$.
Thus adding a set of edges among the $n^2$ possible ones translates into multiplying the generating function with $e^{n^2 w/2}$.
We obtain
\[
	F(z,w) e^z e^{n^2 w/2}.
\]
Finally, we extract the coefficients in $z$ and $w$ to consider only multigraphs with $n$ vertices and $m$ edges and obtain
\[
	n! 2^m m! [z^n w^m] F(z,w) e^z e^{n^2 w/2}.
\]

    \paragraph{Asymptotics.}

We now apply a bivariate saddle-point method to extract the asymptotics (see Section~\ref{sec:saddle-point-heuristic}).

We start from the exact expression~\eqref{eq:distinguished_multi_exact}, to which we apply the following changes of variables
\[
	z \to n z, \quad \text{and} \quad w \to \frac{2m}{n^2} \, w,
\]
and get
\[
	\mg^{[\mF]}_{n,m} =
    n^{2m} \frac{n!}{n^n} \frac{m!}{m^m}
    [z^n w^m]F \left( n z, \frac{2m}{n^2} w \right)
    e^{n z + m w}.
\]
We rewrite the coefficient extractions as Cauchy integrals on circles of radii $1$
\[
	\mg^{[\mF]}_{n,m} =
    n^{2m} \frac{n!}{n^n} \frac{m!}{m^m}
    \frac{1}{(2i\pi)^2}
    \oint_{|w|=1} \oint_{|z|=1}
    F \left( n z, \frac{2m}{n^2} w \right)
    e^{n z + m w}
    \frac{\mathrm dz}{z^{n+1}} \frac{\mathrm dw}{w^{m+1}}.
\]
On this contour 
$\frac{F \left( n z, \frac{2m}{n^2} w \right)}{F \left(n,\frac{2m}{n^2} \right)}$
converges uniformly to $L(z,w)$ which is analytic at $(1,1)$, where its value is~$1$.
Now there exists a sequence of analytic functions $(\epsilon_n(z,w))_{n \geq 0}$
converging uniformly to $0$ such that
\[
	F \left( n z, \frac{2m}{n^2} w \right) =
    F \left(n,\frac{2m}{n^2} \right)
    \left( L(z,w) + \epsilon_n(z,w) \right).
\]
We thus have
\[
	\mg^{[\mF]}_{n,m} \sim 
    n^{2m} \frac{n!}{n^n} \frac{m!}{m^m}
    \frac{1}{(2i\pi)^2}
    \oint_{|w|=1} \oint_{|z|=1}
    F \left(n, \frac{2m}{n^2} \right)
    L(z,w)
    e^{n z + m w}
    \frac{\mathrm dz}{z^{n+1}} \frac{\mathrm dw}{w^{m+1}}.
\]
A simple saddle-point method for large powers at $z=1, w=1$  (see Theorem~VIII.8 from \cite{FS09} or Lemma~\ref{th:large_powers} in the appendix) and Stirling's formula then lead to
\[
	\mg^{[\mF]}_{n,m} \sim 
    n^{2m} \frac{n!}{n^n} \frac{m!}{m^m}
    F \left(n, \frac{2m}{n^2} \right)
    L(1,1)
    \frac{e^n}{\sqrt{2\pi n}}\frac{e^m}{\sqrt{2\pi m}} \sim
    n^{2m} F \left(n, \frac{2m}{n^2} \right).\qedhere
\]
\end{proof}

        \subsection{Simple graphs with one distinguished subgraph}
        \label{sec:distinguished_simple}

We now turn to simple graphs, and establish a result parallel to Theorem~\ref{th:distinguished_multi}, which is valid for multigraphs.

\begin{theorem}[Distinguished, exact/asymptotics, simple] \label{th:distinguished_simple}
i)
The number of simple graphs with $n$ vertices, $m$ edges, and that contain one distinguished subgraph from $\mF$, is
\begin{equation}
	\sg^{[\mF]}_{n,m} =  n! [z^n w^m] F \left(z,\frac{w}{1+w} \right) e^z (1+w)^{\binom{n}{2}} .
\label{eq:distinguished_simple_exact}
\end{equation}

ii)
Assume that the generating function $F$ of the family $\mF$ is entire in both variables.
Let $D$ denote an open set containing the compact unit disk.
Let $n$ and $m = m(n)$ be two integers going to infinity, such that $m = o(n^2)$ and that $\frac{F \big(n z, m w / \binom{n}{2}\big)}{F\big(n, m / \binom{n}{2}\big)}$
converges uniformly on each compact $K \in D^2$ to an analytic function on $D^2$, $L(z,w)$.
Then the set $\mSG^{[\mF]}_{n,m}$  of simple graphs containing $n$ vertices, $m$ edges, and with one copy from $\mF$ distinguished, has weight
\begin{equation}
\label{eq:distinguished_simple_asympt}
    \sg^{[\mF]}_{n,m} \sim
    \binom{\binom{n}{2}}{m}
    F \left( n, \frac{m}{\binom{n}{2}} \right),
\end{equation}
where $\binom{\binom{n}{2}}{m}$ is the total number of graphs with $n$ vertices and $m$ edges.

\end{theorem}

\begin{proof}

The proof mimics the one we gave for multigraphs, and we concentrate our efforts on the points where it differs.

    \paragraph{Exact expression.}
    A simple graph with one distinguished $\mH$-graph is a copy $F$ of a graph from~$\mH$, a set of additional vertices, and a set of additional edges. Those edges can link any pair of vertices, except those already linked in~$F$.
 
We first prove part \emph{i)} of the theorem for a family $\mH$
that is composed of the simple graphs isomorphic to some $(k,\ell)$-graph $H$; we shall then extend it to a general family~$\mF$ of simple graphs.

For the family $\mH$ defined above, let $\aut{H}$ be the number of automorphisms of the graph~$H$.
The number of $H$-graphs is $k!/\aut{H}$, and the generating function of~$\mH$~is 
\[ 
	H(z,w) = \frac{1}{\aut{H}} w^{\ell} z^k.
\]
The generating function of an $H$-graph and a set of isolated vertices is $H(z,w) e^z$.
If we assume that there are $n$ vertices, we extract the coefficient in $z$ and obtain $n! [z^n] H(z,w) e^z$.

\medskip
Then each pair of the $n$ vertices can be linked by an edge, except the pairs already linked in the $H$-graph: the number of edges that can be added is $\binom{n}{2} - \ell$.
For each of those edges, we decide either to add it, or to not add it.
Thus, the generating function of graphs on $n$ vertices with a distinguished $H$-graph, additional vertices, and additional edges, is
\[
    n! [z^n] H(z,w) e^z (1+w)^{\binom{n}{2} - \ell}.
\]
Replacing $H(z,w)$ by its expression, this is equal to
\[
     n! [z^n] \frac{1}{\aut{H}} w^{\ell} z^k e^z (1+w)^{\binom{n}{2} - \ell}
    =
    n! [z^n] H \left(z, \frac{w}{1+w} \right) e^z (1+w)^{\binom{n}{2}}.
\]
Finally, we fix the number of edges to $m$,
and obtain the number of $(n,m)$-graphs
where one $H$-graph is distinguished as
\[
    n! [z^n w^m] H \left(z, \frac{w}{1+w} \right) e^z (1+w)^{\binom{n}{2}}.
\]

\medskip
If $\mF$ is now a general graph family, its generating function is the sum of the generating functions of all $H$-graphs for which there exists at least one isomorphic copy of $H$ in $\mF$:
\[
    F(z,w) =
    \sum_{\text{there is an $H$-graph in $\mF$}}
    H(z,w).
\]
The number of $(n,m)$-graphs where one $\mF$-graph is distinguished is
\[
    \sum_{\text{there is a $H$-graph in $\mF$}}
    n! [z^n w^m] H \left(z, \frac{w}{1+w} \right) e^z (1+w)^{\binom{n}{2}},
\]
and by linearity this is simply
 \[
    n! [z^n w^m] F \left(z, \frac{w}{1+w} \right) e^z (1+w)^{\binom{n}{2}}.
\]

    \paragraph{Asymptotics.}

We again apply a bivariate saddle-point method to extract the asymptotics.
In expression~\eqref{eq:distinguished_simple_exact}, we apply the changes of variables
\[
	z \to n z, \qquad w \to \frac{m}{\binom{n}{2}} w,
\]
and obtain
\[
	\frac{n!}{n^n} \frac{\binom{n}{2}^m}{m^m}
    [z^n w^m]
    F \left(n z, \frac{m w / \binom{n}{2}}{1 + m w / \binom{n}{2}} \right)
    e^{n z} \left( 1 + \frac{m}{\binom{n}{2} w } \right)^{\binom{n}{2}}.
\]
Again the function
\[
    \frac{F \left(nz, m w / \binom{n}{2} \right)}{F\left(n, m / \binom{n}{2} \right)}
\]
converges uniformly to an analytic function $L(z,w)$, with $L(1,1) = 1$.
Furthermore, the function
\[
    \frac{F \left(nz, \frac{m w / \binom{n}{2}}{1 + m w / \binom{n}{2}} \right)}{F\left(n, m / \binom{n}{2} \right)}
\]
converges uniformly to $L(z,w)$ as well,
because $m = \smallo \left(\binom{n}{2} \right)$.
Thus, there exists a sequence of analytic functions $(\epsilon_n(z,w))_{n \geq 0}$
converging uniformly to $0$ such that
\[
	F \left(nz, \frac{m w / \binom{n}{2}}{1 + m w / \binom{n}{2}} \right) =
    F\left(n, \frac{m}{\binom{n}{2}} \right)
    \left( L(z,w) + \epsilon_n(z,w) \right).
\]
The number of $(n,m)$-graphs, with one $\mF$-subgraph distinguished, becomes
\[
	\frac{n!}{n^n} \frac{\binom{n}{2}^m}{m^m}
    F\left(n, \frac{m}{\binom{n}{2}} \right)
    [z^n w^m]
    \left( L(z,w) + \epsilon_n(z,w) \right)
    e^{n z} \left( 1 + \frac{m}{\binom{n}{2}} w \right)^{\binom{n}{2}}.
\]
By an exp-log argument (see Flajolet and Sedgewick~\cite[p.~29]{FS09}), there exists a sequence of analytic functions $(\tilde{\epsilon}_n(w))_{n \geq 0})$ converging uniformly to $0$ such that
\[
	\left( 1 + \frac{m}{\binom{n}{2}} w \right)^{\binom{n}{2}} =
    e^{m w} ( 1 + \tilde{\epsilon}_n(w) ),
\]
so the number of desired graphs can be written as 
\[
	\frac{n!}{n^n} \frac{\binom{n}{2}^m}{m^m}
    F\left(n, \frac{m}{\binom{n}{2}} \right)
    [z^n w^m]
    \left( L(z,w) + \epsilon_n(z,w) \right) (1 + \tilde{\epsilon}(w))
    e^{n z} e^{m w}.
\]
The end of the proof is parallel to the one for Theorem~\ref{th:distinguished_multi}:
the number of graphs is asymptotically equivalent to 
\[
	\frac{n!}{n^n} \frac{\binom{n}{2}^m}{m^m}
    F\left(n, \frac{m}{\binom{n}{2}} \right)
    [z^n w^m]
    L(1,1)
    e^{n z} e^{m w} 
    \sim 
    \binom{\binom{n}{2}}{m}
    F\left(n, \frac{m}{\binom{n}{2}} \right).\qedhere
\]
\end{proof}

        \subsection{Exact enumeration of (multi)graphs with a given number of subgraphs}
        \label{sec:enum-given-number-subgraphs}

We now turn our attention to deriving an exact expression for the number
\[
    \mg^\mF_{n,m,t} = n! 2^m m! [z^n w^m u^t] \mg^\mF(z,w,u)
\]
of multigraphs with $n$ vertices, $m$ edges and containing exactly $t$ $\mF$-subgraphs.
The main difficulty is that these subgraphs may overlap.
To describe the overlaps, we shall use the notion of \emph{patchwork}  defined in Section~\ref{sec:patchworkdef}, see also Section~\ref{sec:patchworks}.

\begin{theorem}[Number of subgraphs, exact, multigraphs] \label{th:exact_mgF}
The number $\mg^\mF_{n,m,t}$ of $(n,m)$-multigraphs that contain exactly $t$ $\mF$-subgraphs is
\begin{equation} \label{eq:multi-subgraphs}
    \mg_{n,m,t}^\mF = 
    n! 2^m m! [z^n w^m u^t]
    \Patch{\mF}(z, w, u-1)
    e^{z}
    e^{n^2 w /2}.
\end{equation}
\end{theorem}

\begin{proof}
The weighted number of $(n,m)$-multigraphs with one $\mF$-patchwork distinguished is 
\[
\mg_{n,m}^{\Patch{\mF}(u)}=n! 2^m m! [z^n w^m] \Patch{\mF}(z, w, u) e^{z} e^{n^2 w /2}.
\]
By Proposition~\ref{th:number_of_subgraphs_general_weights} we have $\mg_{n,m,t}^\mF =[u^t] \mg_{n,m}^{[\Patch{\mF}(u-1)]}$ which completes the proof.
\end{proof}

We defer extracting more detailed information in the case where the subgraphs are strictly balanced until Section~\ref{sec:strictly}, and consider below a result that closely parallels Theorem~\ref{th:exact_mgF}, but now for simple graphs.


\begin{theorem}[Number of subgraphs, exact, simple] \label{th:subgraphs}
The number $\sg_{n,m,t}^\mF$ of $(n,m)$-graphs that contain exactly $t$ $\mF$-subgraphs is
\begin{equation} \label{eq:subgraphs}
    \sg_{n,m,t}^\mF =
    n! [z^n w^m u^t] \Patch{\mF} \left( z, \frac{w}{1+w}, u-1 \right) e^z (1+w)^{\binom{n}{2}}.
\end{equation}
\end{theorem}

\begin{proof}
If we replace $F(z,w)$ by $\Patch{\mF}(z,w,u)$ in Theorem~\ref{th:distinguished_simple} we obtain
\[
    \sg^{\Patch{\mF}(u)}_{n,m} =
    n! [z^n w^m]
    \Patch{\mF} \left(z,\frac{w}{1+w},u \right)
    e^z (1+w)^{\binom{n}{2}}.
\]
Then, applying Proposition~\ref{th:number_of_subgraphs_general_weights} completes the proof.
\end{proof}

	\subsection{Asymptotics for the number of occurrences of subgraphs}
	\label{sec:applications-to-subgraphs}
The problem with the exact expressions derived in Theorems~\ref{th:exact_mgF} and \ref{th:subgraphs} is that we do not know in general the expression of the generating function of patchworks.
Thus, Theorems~\ref{th:distinguished_multi} and~\ref{th:distinguished_simple} cannot be directly applied to extract the asymptotics of the coefficients.
However, partial information is enough to address some interesting problems.
In Section~\ref{sec:densest-subgraphs} below, we give an upper bound for the probability that a high-density subgraph appears in a multigraph of lower density, then derive a bound on the probability of appearance for the densest subgraph of a family~$\cal F$ in a random simple graph.
We then consider in Section~\ref{sec:strictly} the appearance of strictly balanced subgraphs in multigraphs.
Finally, we show in Section~\ref{sec:distribution_number_subgraphs} that we can easily rederive the well-known Poisson limiting distribution for the number of subgraphs in a simple graph.

	\subsubsection{Dense subgraphs, and densest subgraph}
        \label{sec:densest-subgraphs}

We investigate here whether a subgraph of high density is likely to appear in a (multi)graph of smaller density.

As a first application of Theorem~\ref{th:distinguished_multi}, we prove that subgraphs with high edge density are unlikely to appear in random multigraphs with small edge density.
This result was first derived by Erd\H{o}s and R\'enyi~\cite{ER60}.

\begin{corollary}[Probability for high-density subgraph, multigraphs] \label{th:general_small_density}
Assume that $m = \smallo \left( n^{2 - 1/d(F)} \right)$. 
Then the probability for a uniform random multigraph from $\mMG_{n,m}$
to contain one or more copies of the subgraph $F$ 
tends to $0$ at rate $\bigO \left( m \, n^{1/d(F) - 2} \right)^{m(F)}$, as $n\to\infty$.
\end{corollary}

\begin{proof}
Let $F$ be a given multigraph with generating function
\[
	F(z,w)=\frac{z^{n(F)}}{n(F)!}\frac{w^{m(F)}}{2^{m(F)}m(F)!}.
\]
By Proposition~\ref{th:general_weight_mean} the expected number of $F$-subgraphs in a random $(n,m)$-multigraph is $\esp(G[F])=\mg_{n,m}^{[\mF]}/\mg_{n,m}=\mg_{n,m}^{[\mF]}/n^{2m}$. 
Applying Theorem~\ref{th:distinguished_multi}, we obtain for $m = \bigO(n^{\alpha})$ with $0 < \alpha < 2$
\[
	\esp(G[F]) \sim F\left(n,\frac{2m}{n^2}\right) = \bigO(n^{n(F) - 2m(F) +\alpha m(F)}) .
\]
Finally, this yields
\[
	\alpha < 2- \frac{1}{d(F)} \Rightarrow \esp(G[F]) \xrightarrow[n\to \infty]{} 0 \Rightarrow \mathbb{P}(G[F] > 0) \xrightarrow[n \to \infty]{} 0. \qedhere
\]
\end{proof}

This bound can be further improved by noticing that the relevant parameter is not the density of the multigraph $F$, but the density of its \emph{densest} subgraph, \ie its \emph{essential density} $d^\star(F)$.
We directly obtain the following:

\begin{corollary}[Probability for high-essential-density subgraph, multigraphs] 
\label{cor:essential_density_multigraph}
Denote by $d^{\star}$ the density of a maximal densest subgraph of $F$, and consider a random $(n,m)$-multigraph $G$ with $m = \bigO(n^{\alpha})$ for some fixed $0 < \alpha < 2$.
Then
\[
\begin{cases}
	G[F] = 0 \textrm{ asymptotically almost surely} & \textrm{ if } \alpha < 2 - 1/d^{\star}, \\
    \esp(G[F]) = \bigO(n^{m(F) (\alpha - 2 + 1/d)}) & \textrm{ if } \alpha \ge 2 - 1/d^{\star} .     
\end{cases}
\]
\end{corollary}

\begin{proof}
Let $G$ denote a random $(n,m)$-multigraph.
If $H$ is a subgraph of $F$, then $G$ contains $F$ only if it contains $H$, so
\[
	\proba(G[F] > 0) \leq \proba(G[H] > 0).
\]
Assume that $H$ is a densest subgraph of $F$. By Corollary~\ref{th:general_small_density}, $H$ almost surely does not appear whenever $\alpha < 2 - 1/d(H) = 2 - 1/d^{\star}$, and the same holds for $F$.

The bound on $\esp(G[F])$ when $\alpha \ge 2 - 1/d^{\star}$ follows directly from Theorem~\ref{th:distinguished_multi}.
\end{proof}

We now turn to simple graphs, for which we obtain a new proof of the following classic result of Erd\H{o}s and R\'enyi~\cite{ER60} and Bollob\'as~\cite{Bo81} by a simple application of Theorem~\ref{th:distinguished_simple}.

\begin{corollary}[Probability for high-essential-density subgraph, simple graphs] 
\label{cor:prob-densest-simple}
Denote by $d^{\star}$ the density of a maximal densest subgraph of $F$, and consider a random $(n,m)$-graph $G$ with $m = \bigO(n^{\alpha})$ for some fixed $0 < \alpha < 2$.
Then
\[
\begin{cases}
	G[F] = 0 \textrm{ asymptotically almost surely} & \textrm{ if } \alpha < 2 - 1/d^{\star}, \\
    \esp(G[F]) = \bigO(n^{m(F) (\alpha - 2 + 1/d)}) & \textrm{ if } \alpha \ge 2 - 1/d^{\star} .     
\end{cases}
\]
\end{corollary}

\begin{proof}
We use the results from Theorem~\ref{th:distinguished_simple} and the same arguments as in the proofs of Corollaries~\ref{th:general_small_density} and~\ref{cor:essential_density_multigraph}.
\end{proof}

        \subsubsection{Strictly balanced subgraphs in multigraphs}
        \label{sec:strictly}

We have already noted that extracting exact information on the number of subgraphs from the exact expression derived in Theorem~\ref{th:exact_mgF} is hindered by the fact that we do not know an explicit expression for the generating function of patchworks.
In the case of strictly balanced subgraphs we can nevertheless obtain some information on the distribution of the number of occurrences of these subgraphs, and show that it follows asymptotically a Poisson distribution.

For the notions we use in the context of patchworks recall Section~\ref{sec:patchworks}. Moreover, since patchworks are weighted multigraphs, enriched with further information, we define its density $d(P)$ and its essential density $d^\star(P)$ similarly as for multigraphs. If $P$ is a $F$-patchwork and has essential density $d^\star(F)$, we call it an \emph{essential patchwork}.
Recall also that a patchwork is disjointed if it contains no pair of pieces sharing one vertex or more.
Given a multigraph $F$, let $\pair_F$ denote the set of $F$-patchworks that contain exactly two pieces and are not disjointed.
The following lemma is due to  Erd\H{o}s and R\'enyi~\cite{ER60}; we give again its proof for the sake of completeness.

\begin{lemma} \label{th:strict_balanced}
A multigraph $F$ is strictly balanced if and only if
the density of any patchwork from $\pair_F$
is greater than the density of $F$.
\end{lemma}

\begin{proof}
We consider a patchwork $P$ from $\pair_F$, denote by $F_1$ and $F_2$ its two pieces, copies of $F$, and by $H$ the maximum common subgraph of $F_1$ and $F_2$.
Let also $I$ denote the couple $(V(F_1) \setminus V(H), E(F_1) \setminus E(H))$.
$I$ is not a multigraph, because some of its edges link a vertex from $I$ to a vertex from $H$.
However, we can define its density as
$$d(I) = \frac{m(F_1) - m(H)}{n(F_1) - n(H)}.$$
The multigraph $F_1$ is obtained by adding the vertices and edges from $I$ to $H$.
Since $F_1$ is strictly balanced, we have $d(H) < d(F_1)$, which implies $d(I) > d(F_1)$.
The patchwork $P$ from $\pair_F$ is obtained by adding the vertices and edges from $I$ to $F_2$.
Thus the density of $P$ is greater than the density of $F$.
\end{proof}

In particular, Lemma~\ref{th:strict_balanced} entails that only the disjointed patchworks of a strictly balanced multigraph are essential, and Corollary~\ref{th:general_small_density} implies that they are the most likely to appear.

\medskip
We now consider the limit law of $G[F]$, the number of $F$-subgraphs in $G$, when $F$ is a fixed strictly balanced multigraph, and $G$ is drawn uniformly at random from $\mMG_{n,m}$, for $m = m(n)$.
According to Corollary~\ref{th:general_small_density}, a random multigraph from $\mMG_{n,m}$
with $m = \smallo \left( n^{2 - n(F)/m(F)} \right)$ is unlikely to contain any copy of the multigraph $F$.
We now focus on the critical case $m = \exactbigO\left( n^{2 - n(F)/m(F)} \right)$.

\begin{theorem}[Poisson law, strictly-balanced, multigraphs]
\label{th:limit_law_multi}
Given a strictly balanced multigraph $F$,
setting $\alpha = 2 - 1 / d(F)$
and $m \sim c n^{\alpha}$,
and considering a random multigraph $G$ from $\mMG_{n,m}$,
then $G[F]$ follows in the limit a Poisson law of parameter
\[
    \lambda(F) = \frac{c^{m(F)}}{m(F)! \, n(F)!}.
\]
\end{theorem}

\begin{proof}
Our goal is to get the result by applying Proposition~\ref{th:general_weights_second_asymptotics}. Thus for a multigraph $F$ let 
\begin{equation}\label{G_Gleichung}
F(z,w)=\frac{z^{n(F)}}{n(F)!}\frac{w^{m(F)}}{2^{m(F)}m(F)!}
\end{equation}
and set $\lambda_n(F)=F\left(n,\frac{2m}{n^2}\right)$. Then 
\[
    \lambda(F) = \lim_{n \to \infty}
    F\left(n, \frac{2m}{n^2} \right)
    =
    \frac{c^{m(F)}}{m(F)!\, n(F)!}
\]
which is clearly a positive number. Thus the first condition of Proposition~\ref{th:general_weights_second_asymptotics} is satisfied. 

To show that the second condition of Proposition~\ref{th:general_weights_second_asymptotics} is satisfied, take an arbitrary multigraph $G$ and some constant $\kappa$ and observe that 
\[
\mg_{n,m}^{[Ge^{\kappa F}]} = n!m!2^m [z^nw^m] G(z,w) e^{\kappa F(z,w)} e^z e^{n^2w/2}
\]
where $F(z,w)$ is the function given in \eqref{G_Gleichung} and $G(z,w)$ the analogue for the multigraph $G$. 
By Theorem~\ref{th:distinguished_multi} $\mg_{n,m}^{[Ge^{\kappa F}]}$ can be expressed as 
\[
n!m!2^m [z^nw^m] G(z,w) e^{\kappa F(z,w)} e^z e^{n^2w/2} \sim n^{2m} G\left(n,\frac{2m}{n^2}\right) \exp\left( \kappa F\left(n,\frac{2m}{n^2}\right)\right) = \mg_{n,m} \lambda_n(G) e^{\kappa\lambda_n(F)}
\]
which matches the requirement from Proposition~\ref{th:general_weights_second_asymptotics}. The only task which is left is showing that $\lambda_n(G)$ tends to zero if $d(G)>d(F)$. Looking at \eqref{G_Gleichung} again, we see that (noticing that $\alpha=2-\frac1{d(F)}$) the exponent of $n$ in 
\[
G\left(n,\frac{2m}{n^2}\right) = \frac{c^{m(G)}}{n(G)!m(G)!} n^{n(G)-\frac{m(G)}{d(F)}}
\]
is negative whenever $d(G)>d(F)$. Thus Proposition~\ref{th:general_weights_second_asymptotics} is applicable and yields directly the assertion.
\end{proof}

	\subsubsection{Distribution of the number of occurrences of a subgraph in a simple graph}
	\label{sec:distribution_number_subgraphs}

We now turn to getting information on the distribution of the number of occurrences of a subgraph in a simple graph.
The following theorem was first derived by Bollob\'as~\cite{Bo81}.

\begin{theorem}[Poisson law, strictly-balanced, simple] 
\label{th:limit_law_simple}

Let $F$ denote a strictly balanced graph of density $d$, with $\ell$ edges and $\aut{F}$ automorphisms, and assume $m \sim c n^{2-1/d}$ for some positive constant~$c$.
Then the number of $F$-subgraphs in a random $(n,m)$-graph $G$ follows a Poisson limit law with parameter $\lambda(F) = (2c)^{\ell}/\aut{F}$, \ie for any nonnegative integer $t$,
\[
    \lim_{n \to \infty} \proba(G[F] = t) = \frac{\lambda(F)^t}{t!} e^{-\lambda(F)}.
\]
\end{theorem}

\begin{proof}
The result can be shown by a reasoning similar to the one presented in the proof of Theorem~\ref{th:limit_law_multi}. As in the proof of Proposition~\ref{th:general_weights_second_asymptotics} it can be shown that the main contribution in~\eqref{eq:subgraphs} at the threshold $\alpha= 2- 1/d$ comes from disjointed patchworks. Thus an analogue of Proposition~\ref{th:general_weights_second_asymptotics} for simple graphs is true as well. 

Setting $\lambda_n(G):=G\left(n,m/\binom n2\right)$, using $\lambda(F)=\lim\limits_{n\to \infty} F\left(n, m/\binom{n}{2}\right)$ and the asymptotics provided in Theorem~\ref{th:distinguished_simple} completes the proof.
\end{proof}

\subsubsection{Asymptotics of $F$-free (multi)graphs}

Given a sub(multi)graph $F$, let an \emph{$F$-free (multi)graph} be a (multi)graph without $F$-subgraph. 
The two preceding theorems allow in particular the enumeration of $F$-free (multi)graphs for any strictly balanced $F$ at the threshold $m\sim cn^{1-1/d(F)}$.

\begin{corollary}[$F$-free multigraphs]
	Let $F$ denote a strictly balanced multigraph of density $d$ and assume $m \sim c n^{2-1/d}$ for some positive constant~$c$.
Then the number of $F$-free $(n,m)$-multigraphs is given by 
\[
    \mg^{F}_{n,m,0} = e^{-\lambda(F)} \mg_{n,m}, \quad\textrm{ with }\lambda(F) := \frac{c^{m(F)}}{n(F)!m(F)!}.
\]
\end{corollary}

\begin{corollary}[$F$-free simple graphs]
	Let $F$ denote a strictly balanced graph of density $d$, and $\ell$ edges, and assume $m \sim c n^{2-1/d}$ for some positive constant~$c$.
Then the number of $F$-free $(n,m)$-graphs is given by 
\[
    \sg^{F}_{n,m,0} = e^{-\lambda(F)} \sg_{n,m}, \quad\textrm{ with }\lambda(F) := \frac{(2c)^{\ell}}{\aut{F}}.
\]
\end{corollary}


		\section{Small subgraphs in multigraphs weighted according to their degrees}
        \label{sec:results-degree-constraints}


We begin with a presentation of our model for multigraphs with degree constraints in Section~\ref{sec:degree-constraints}, then making explicit its relation to the configuration model in Section~\ref{sec:config-Boltzmann}.
Section~\ref{sec:degrees-exact} then gives an exact expression for the total weight of multigraphs with one specified subgraph, from which we derive the expected number of subgraphs.
In the following sections we study the asymptotics of this expected number.
We first consider the case where only a finite number of weights is nonzero (Section~\ref{sec:finite-set-weights}),
then turn to the case when the weights decrease ``fast enough'' with the degrees (Section~\ref{sec:infinite-set-weights}).
There we give a general result (Lemma~\ref{th:distinguished_degree_constraints-asymptotic_infinite-set}) that holds under quite general conditions and from which we obtain various results on the appearance of some classes of subgraphs.
The consideration of infinite families of weights that do not satisfy the conditions of Lemma~\ref{th:distinguished_degree_constraints-asymptotic_infinite-set} is deferred to Section~\ref{sec:failed-assumptions}.

\subsection{Model for multigraphs weighted according to their degrees}
\label{sec:degree-constraints}
 
In many applications, random multigraphs are not necessarily sampled according to the uniform distribution. Instead, we now want to be able to control the degree distribution of a random multigraph. To that effect, we place weights on vertices of a given degree. More formally, to an infinite sequence $\vdelta = (\delta_0, \delta_1, \dots)$ of nonnegative weights,
we associate the following generating function $\Delta(x)$
\[
	\Delta(x) = \sum_{d\ge 0} \delta_d \frac{x^d}{d!}.
\]
Define $\mF_{\Delta}$ as the multigraph family $\mF$
equipped with the following weight:
each vertex of degree $d$ carries a weight $\delta_d$,
and the weight of a multigraph $G$ is the product of the weights of its vertices
\[
	\omega_{\Delta}(G) = \prod_{v\in V(G)} \delta_{\deg(v)} = \prod_{d \ge 0} \delta_d^{n_d(G)},    
\]
where $n_d(G)$ stands for the number of vertices of $G$ with degree $d$.
The total weight of $(n,m)$-multigraphs in $\mF_{\Delta}$ is denoted by $F_{n,m,\Delta}$,
and the generating function of $\mF_{\Delta}$ is
\[
	F_{\Delta}(z,w) = \sum_{G \in \mF} \omega_{\Delta}(G) \frac{w^{m(G)}}{2^{m(G)} m(G)!} \frac{z^{n(G)}}{n(G)!}.
\]
The $\Delta$-multigraphs are the weighted multigraphs from $\mMG_{\Delta}$,
and the \emph{$(n,m,\Delta)$-multigraphs} are
the $\Delta$-multigraphs with $n$ vertices and $m$ edges.
In particular, the total weight $\mg_{n,m,\Delta}$
of the weighted multigraph family $\mMG_{n,m,\Delta}$ is then
\[
	\mg_{n,m,\Delta} =
    \sum_{G\in \mMG_{n,m,\Delta}} \omega_{\Delta}(G) =
    \sum_{G\in \mMG_{n,m,\Delta}} \prod_{v\in V(G)} \delta_{\deg(v)}.
\]

In this weighted model, it will often prove convenient to work with an additional variable $x$ for (labeled) half-edges of the multigraph. Indeed, since  each edge is labeled and oriented, it is represented as a triple $(u,v,e)$ where $u$ and $v$ denote its endpoints, and $e$ its label. This edge can then be cut into two labeled half-edges, one hanging from $u$ with label $2e-1$, the other hanging from $v$ with label $2e$. Cutting all the edges of a multigraph into half-edges, we obtain a bijection between the $(n,m)$-multigraphs and the set of $n$ labeled vertices, each vertex of degree $d$ having $d$ dangling labeled half-edges around it, and so that the total number of half-edges sums up to $2m$. 
This construction, which is reminiscent of the \emph{configuration model} of Bollob\'as~\cite{Bollobas80} and of Wormald~\cite{Wo78} (we explore further the link with this model in Section~\ref{sec:config-Boltzmann} below), yields the following lemma, first derived in \cite{EdPR16}:

\begin{lemma}\label{lem:total_weight}
The total weight of $\mMG_{n,m,\Delta}$ is equal to
\[
  	\mg_{n,m,\Delta} = (2m)![x^{2m}] \Delta(x)^n,
\]
and the generating function of $\mMG_{\Delta}$ is equal to
\[
	\mg_{\Delta}(z,w) =
    \sum_{m\ge 0} (2m)! [x^{2m}] e^{z\Delta(x)} \frac{w^m}{2^m m!}.
\]
\end{lemma}

\begin{proof}
The generating function of sets of labeled half-edges is $\Delta(x)$, 
where the set of size $d$ has weight $\delta_d$.
Consider the combinatorial family
composed of sets of vertices such that
\begin{itemize}
\item
to each vertex is attached a set of labeled half-edges,
\item
the weight of the vertex is $\delta_d$
if $d$ half-edges are attached to it,
\item
the weight of the object is the product of the weights of its vertices.
\end{itemize}
Then the generating function of this combinatorial family is $e^{z \Delta(x)}$.
The sum of the weights of the objects from this family
that contain $2m$ half-edges in total is then
\[
	(2m)! [x^{2m}] e^{z \Delta(x)}.
\]
Such objects are in bijection with $\Delta$-multigraphs containing $m$ edges,
according to the discussion before the statement of the lemma.
The generating function of a set of $m$ labeled edges is $\frac{w^m}{2^m m!}$.
Summing over $m$, we obtain the generating function of $\Delta$-multigraphs
\[
	\mg_{\Delta}(z,w) =
    \sum_{m\ge 0} (2m)! [x^{2m}] e^{z\Delta(x)} \frac{w^m}{2^m m!}.\qedhere
\]
\end{proof}

	\subsection{Random $\Delta$-multigraphs and the link with the configuration model.}
    \label{sec:config-Boltzmann}
    
Throughout this section, we compare the distribution of weighted multigraphs given by weighted Boltzmann samplers, denoted by $\mB$, coming from analytic combinatorics, and by the purely probabilistic configuration model, denoted by $\mC$.

In our weighted model, a random $(n,m,\Delta)$-multigraph
is no longer chosen uniformly at random from $\mMG_{n,m}$,
but rather according to its weight:
\[
	\proba(G \in \mMG_{n,m,\Delta}) = \frac{\omega_{\Delta}(G)}{\mg_{n,m,\Delta}}.
\]

\paragraph{Weighted Boltzmann sampling.} One can effectively sample according to this distribution
thanks to weighted Boltzmann samplers, introduced by Duchon \emph{et al.}~\cite{DFLS04}.
Let us give some classical properties of this random sampler.
Define the set $V_\Delta := \{ v_{(d)}\}_{d \ge 0}$, where $v_{(d)}$ is a single vertex with label 1, and $d$ pending half-edges labeled from $1$ to $d$. Let the variable $x$ count the number of half-edges, its weighted generating series is thus $\Delta(x)$.
Given a real positive value $x$,
let $\Gamma_x^{\cal B} V_\Delta$ denote the random sampler for $V_\Delta$
that outputs $v_{(d)}$ with a probability equal to
\[
	\proba(\Gamma_x^{\cal B} V_\Delta = v_{(d)}) = \frac{\delta_d }{\Delta(x)} \frac{x^d}{d!}.
\]

Define the weighted Boltzmann sampler $\Gamma_x^{\cal B} V_\Delta^n$, for the weighted family $\mMG_{n,\Delta} := \bigcup\limits_{m \ge 0} 
\mMG_{n,m,\Delta} $ as follows:
\begin{enumerate}
	\item Make $n$ independent calls to the sampler $\Gamma_x^{\cal B} V_\Delta$ to produce $n$ vertices labeled from $1$ to $n$, with pending labeled half-edges,
    \item If the sum of the degrees is odd, reject and repeat step 1. Otherwise, set $2m$ as the sum of the degrees.
    \item Choose uniformly at random a relabeling of the half-edges such that the labels range from $1$ to $2m$,
    \item For all $j=1\dots m$, create a new directed edge labeled by $j$ by connecting the half-edge with label $2j-1$ to the half-edge with label $2j$.
\end{enumerate}
Hence this sampler outputs a $(n,m,\Delta)$-multigraph $G$ with degree sequence $(d_1,\dots,d_n)$ with probability 
\begin{align*}
	\proba\left(\Gamma_x^{\cal B} V_\Delta^n = G\right) &= 0 \quad \textrm{ if } \sum_i d_i \textrm{ is odd and } \\
	\proba\left(\Gamma_x^{\cal B} V_\Delta^n = G\right) &= 
    \proba\left(v_1 = d_1, \dots v_n =d_n\right)\cdot 
    \proba\left(\Gamma_x^{\cal B} V_\Delta^n = G |v_1 = d_1, \dots v_n =d_n\right) \cdot
    \proba\left( \sum_i d_i \textrm{ even}\right)^{-1} 
    \\
	&= \left(\prod_{i=1}^n \frac{\delta_{d_i} }{\Delta(x)} \frac{x^{d_i}}{d_i!} \right) \cdot 
    {2m \choose d_1,\dots,d_n}^{-1} \cdot 
    \left( \frac{ \Delta(x)^n + \Delta(-x)^n}{2 \Delta(x)^n}\right)^{-1} 
    \\
    &= \frac{2}{\Delta(x)^n + \Delta(-x)^n} \frac{x^{2m}}{(2m)!} \omega_\Delta(G) \quad \textrm{ if } \sum_i d_i \textrm{ is even}.
\end{align*}

When conditioned on having $m$ edges, we obtain a weighted Boltzmann sampler for the weighted family $\mMG_{n,m,\Delta}$ which outputs a $(n,m,\Delta)$-multigraph $G$ with degree sequence $(d_1,\dots,d_n)$ with probability 
\begin{align*}
	\proba(\Gamma_x^{\cal B} V_\Delta^n = G | \Gamma_x^{\cal B} V_\Delta^n \in \mMG_{n,m,\Delta}) 
    &= 
    \frac{\proba(\Gamma_x^{\cal B} V_\Delta^n = G)}{\proba( \Gamma_x^{\cal B} V_\Delta^n \in \mMG_{n,m,\Delta})} 
    = 
    \frac{\proba(\Gamma_x^{\cal B} V_\Delta^n = G)}{\sum_{G' \in\mMG_{n,m,\Delta}} \proba( \Gamma_x^{\cal B} V_\Delta^n = G')}
    \\
	&= \left(\frac{2 \omega_\Delta(G)}{\Delta(x)^n + \Delta(-x)^n}\frac{x^{2m}}{(2m)!} \right) \cdot \left( \sum_{G' \in \mMG_{n,m,\Delta}} \frac{2 \omega_\Delta(G') }{\Delta(x)^n + \Delta(-x)^n}\frac{x^{2m}}{(2m)! }\right)^{-1}\\
    &= \frac{\omega_\Delta (G)}{\sum_{G' \in \mMG_{n,m,\Delta}} \omega_\Delta(G') } =  \frac{\omega_{\Delta}(G)}{\mg_{n,m,\Delta}},
\end{align*}
as desired.

In the unconditioned sampler the number of edges $m$ is a random variable whose distribution depends on the parameter $x$.
If we condition on a particular value $m$
(\ie we start over if the number of edges obtained is not $m$),
then the probability to sample a given $(n,m,\Delta)$-multigraph
is proportional to its weight. 
More precisely, one has
\[
	\esp( \deg{\Gamma_x^\mB V_\Delta} ) = \frac{x \Delta'(x)}{\Delta(x)},
\]
\[
	\mathbb{V}( \deg{\Gamma_x^\mB V_\Delta}  ) = 
    \frac{x^2 \Delta''(x)}{\Delta(x)}
    + \frac{x \Delta'(x)}{\Delta(x)}
    - \left(\frac{x \Delta'(x)}{\Delta(x)}\right)^2.
\]

In order to target $m$ edges, we tune the value $x$ such that the total expected number of half-edges is $2m$, so each vertex receives in average $2m/n$ half-edges.
Because the mean degree of $\Gamma_x^{\mathcal{B}} V_\Delta$ is $\frac{x \Delta'(x)}{\Delta(x)}$,
this corresponds to choosing for $x$ the positive solution of the relation
\begin{equation}
	\label{eq:x-boltzmann}
	\frac{x \Delta'(x)}{\Delta(x)} = \frac{2m}{n}.
\end{equation}
Let $\support(\Delta)$ denote the set of indices
corresponding to nonzero weights
\[
	\support(\Delta) = \{d\ |\ \delta_d > 0\}.
\]
Observe that the equation \eqref{eq:x-boltzmann} has a unique positive solution whenever
\[
	\min(\support(\Delta)) < \frac{2m}{n} < \max(\support(\Delta)),
\]
because the average value of a random variable stays between
its lowest and largest possible values,
and because the function $\frac{x \Delta'(x)}{\Delta(x)}$
is increasing, since its derivative is positive as it equals the variance of $\deg \Gamma_{x}^\mB V_\Delta$ up to a (positive) factor $x$.

\paragraph{Configuration model.} The configuration model has been first introduced by Bollob\'as~\cite{Bollobas80} to sample multigraphs according to a prescribed degree sequence; see also the presentation given by van der Hofstad~\cite[Ch.~7]{vdHof16}. It was later extended to sample multigraphs with a given degree probabilistic distribution. Let us recall briefly the process.

Given a probability distribution $\pi$ over the nonnegative integers, a configuration model sampler $\Gamma_{\pi}^{\mathcal{C}}\mMG_n$ produces a multigraph $G=(V,E)$ with $n$ vertices as follows:
\begin{enumerate}
	\item Draw $n$ independent random variables $X_1,\dots,X_n$ with distribution $\pi$.
    \item Let $\sum X_i := M$: if $M$ is odd, reject and repeat step 1, otherwise let $M := 2m$.
    \item Let $V=\{ v_1, \dots, v_n\}$, where $v_i$ has $X_i$ unordered pending half-edges, and thus $\deg(v_i) = X_i$.
    \item Pick uniformly at random the labels of each half-edges from the set $\{1,\dots,2m\}$.
    \item For all $j = 1 \ldots  m$, pair the half-edge of label $2j-1$ with the half-edge of label $2j$, the new (directed) edge carries label $j$.
\end{enumerate}

\begin{proposition}
Suppose that, for a given real value $x$, the probability distribution $\pi_x$ is defined by:
\[	
	\forall d\ge 0:\quad \pi_x(d) = \frac{\delta_d x^d}{\Delta(x) d!}.
\]
The distributions over $\mMG_n$ under the configuration model with parameter $\pi_x$ and the weighted Boltzmann model with parameter $x$ are equal. More precisely, given a multigraph $G\in \mMG_n$, where $\deg(v_i) =d_i, 1\le i \le n$, we have:
\[ 
	\forall G \in \mMG_n:\quad \proba(\Gamma^{\mathcal{C}}_{\pi_x} \mMG_n = G ) = \proba(\Gamma^{\mathcal{B}}_x V^n_\Delta = G).
\]
\end{proposition}

\begin{proof}
Let $G$ be a $(n,m)$-multigraph with vertices $v_1,\dots,v_n$, where $\deg(v_i) =d_i (1\le i \le n)$, we have
\begin{align*}
	\proba(\Gamma^{\mathcal{C}}_{\pi _x} \mMG_n = G ) 
    &= \proba(X_1 = d_1, \dots, X_d=d_n)\cdot 
    \proba(\Gamma^{\mathcal{C}}_{\pi _x} \mMG_n = G | X_i = d_i, \forall 1\le i \le n) \cdot
    \proba\left(\sum_i d_i \textrm{ even}\right)^{-1}\\
    &= \prod_{i=1}^n \frac{\delta_{d_i} x^{d_i}}{\Delta(x) d_i!} \cdot 
    {2m \choose d_1,\dots, d_n}^{-1} \cdot
    \left( \frac{ \Delta(x)^n + \Delta(-x)^n}{2 \Delta(x)^n}\right)^{-1} 
    = \proba\left(\Gamma^{\mathcal{B}}_x V^n_\Delta = G\right).\qedhere
\end{align*}
\end{proof}

	\subsection{Exact number of subgraphs}
    \label{sec:degrees-exact}

We consider in this part how to extend the exact and asymptotic enumeration results of Section~\ref{sec:subgraphs} to take into account the degree constraints we presented in Section~\ref{sec:degree-constraints}.
Theorem~\ref{th:distinguished_degree_constraints-symbolic} gives an exact expression for the total weight of multigraphs with one distinguished subgraph in the family~$\mF$.
This allows us to obtain the expected number of subgraphs belonging to~$\mF$ and the probability that there are $t$ such subgraphs (Corollaries~\ref{th:deg_constraints_mean} and~\ref{th:deg_constraints_exact_subgraphs}).
 
It will prove convenient to add variables to the generating function of multigraph families
in order to mark the degrees.
Each vertex of degree $d$ is marked by the variable $y_d$,
and the infinite sequence of variables $(y_1, y_2, \ldots)$ is denoted by $\vy$.
The generating function of a multigraph family $\mF$ is then
\begin{equation} \label{eq:gf_deg}
	F(z,w,\vy) = \sum_{G \in \mF} \bigg( \prod_{v \in V(G)} y_{\deg(v)} \bigg) \frac{w^{m(G)}}{2^{m(G)} m(G)!} \frac{z^{n(G)}}{n(G)!}.
\end{equation}
Likewise, the generating function of a simple graph family $\mF$ is
\[
	F(z,w,\vy) = \sum_{G \in \mF} \bigg( \prod_{v \in V(G)} y_{\deg(v)} \bigg) w^{m(G)} \frac{z^{n(G)}}{n(G)!}.
\]
Observe that the generating function of the (multi)graphs from $\mF_{\Delta}$
is equal to the generating function of $\mF$ where $\vy$ is replaced by $\vdelta$:
\[
	F_{\Delta}(z,w) = F(z,w,\vdelta).
\]
The variables $z$ and $w$ are redundant in the generating function $F(z,w,\vy)$,
because each vertex has a degree,
and the sum of the degrees is twice the number of edges.
Thus, given two variables $a$, $b$, we have formally
\begin{equation} \label{eq:Fzwy}
	F(a z, b w, \vy) = F\left(z, w, \left(a b^{d/2} y_d \right)_{d \geq 0} \right)
\end{equation}
Notice that the sum of degrees being even, it ensures that the exponents of $b$ are non-negative integers, defining a proper formal power series.

\begin{theorem}[Distinguished, total weight, multigraphs]
\label{th:distinguished_degree_constraints-symbolic}
Define the operator $ \bar{\partial} $ as
\[
	\bar{\partial} f (x) = (f(x), f' (x), f'' (x), \ldots).
\]
Given a multigraph family $\mF$ with generating function $F(z,w,\vy)$
(with the convention stated in Equation~\eqref{eq:gf_deg}),
the total weight $\mg^{[\mF]}_{n,m,\Delta}$ of $(n,m,\Delta)$-multigraphs
where one $\mF$-subgraph is distinguished
is equal to
\begin{equation} \label{eq:distinguished_degree_constraints}
	\mg^{[\mF]}_{n,m,\Delta} =
	n! 2^m m! [z^n w^m]
    \sum_{j \geq 0}
    (2j)! \, [x^{2j}]
    F \left(z,w,\bar{\partial} \Delta (x) \right)
    e^{z \Delta(x)}
    \frac{w^j}{2^j j!}.
\end{equation}
\end{theorem}

\begin{proof}
We combine the half-edges construction of Section~\ref{sec:degree-constraints}
with the proof of Theorem~\ref{th:distinguished_simple}.
Let $H$ denote a multigraph from $\mF$,
with $k$ vertices and $\ell$ edges,
and assume its automorphism group (both on vertices and edges) has size $\aut{H}$.
Then the number of $H$-multigraphs is $2^{\ell} \ell! k! / \aut{H}$,
and the generating function of the $H$-multigraphs~is
\[
	H(z,w,\vy) =
    \frac{1}{\aut{H}} \bigg(\prod_{v \in H} y_{\deg(v)} \bigg) w^k z^{\ell}.
\]
An $(n,m,\Delta)$-multigraph $G$ where an $H$-subgraph is distinguished
can be uniquely decomposed as an $H$-multigraph,
a set of additional vertices,
and a set of labeled half-edges, each linked to a vertex.
The total number of half-edges must be even, say $2 j$.
The weight of any vertex linked to $d$ half-edges and $e$ edges in $H$ is then $\delta_{d+e}$,
and the weight of $G$ is equal to the product of the weights of the vertices. Equivalently, a vertex (without weight) of degree $d$ in $H$ is substituted by a vertex of arbitrary degree (which is at least $d$) with adequate weight and with $d$ distinguished half-edges (discounted because already counted by the variable $w$ in $H(z,w,\cdot)$).
\begin{figure}\label{fig:add_half}
\centering
\includegraphics[width=\linewidth]{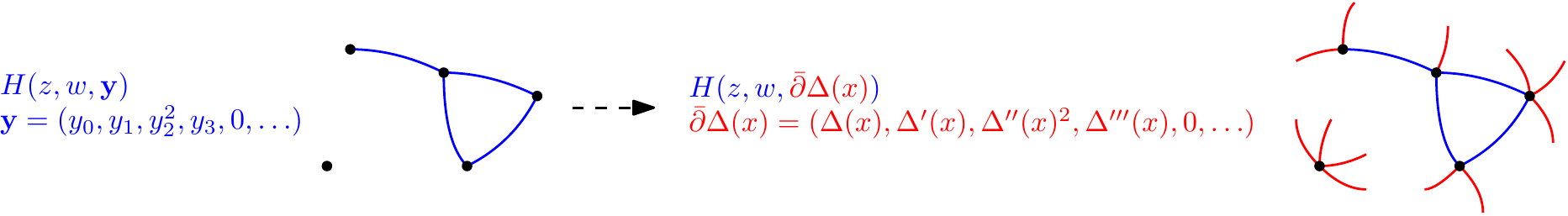}
\caption{A distinguished graph $H$ with some additional half-edges.}
\end{figure}

The generating function of $\Delta$-multigraphs,
where one $H$-subgraph is distinguished, is then
\[
	\mg_{\Delta}^{[\{H\}]}(z,w) =
	\sum_{j \geq 0} (2j)! \,
    [x^{2j}] H(z,w, \bar{\partial} \Delta (x))\,
    e^{z \Delta(x)} \frac{w^j}{2^j j!}.
\]
Using the decomposition
\[
    F(z,w,\vy) =
    \sum_{\text{there is an $H$-graph in $\mF$}}
    H(z,w,\vy).
\]
and extracting the coefficient $n! 2^m m! [z^n w^m]$ concludes the proof.
\end{proof}

We easily derive from Theorem~\ref{th:distinguished_degree_constraints-symbolic} information on the number of occurrences of subgraphs from~$\mF$.

\begin{corollary}[Expected number of subgraphs, degree constraints, multigraphs]
\label{th:deg_constraints_mean}
Given a multigraph family $\mF$,
the expected number of $\mF$-subgraphs in a random $(n,m,\Delta)$-multigraph
is
\[
	\frac{\mg_{n,m,\Delta}^{[\mF]}}{\mg_{n,m,\Delta}} =
    \frac{n! 2^m m! [z^n w^m]
    \sum_{j \geq 0}
    (2j)! \, [x^{2j}]
    F \left(z,w,\bar{\partial} \Delta (x) \right)
    e^{z \Delta(x)}
    \frac{w^j}{2^j j!}}
    {(2m)! [x^{2m}] \Delta(x)^n}.
\]
\end{corollary}

\begin{proof}
According to Proposition~\ref{th:general_weight_mean},
the average number of $\mF$-subgraphs in a random $(n,m,\Delta)$-multigraph is
\[
	\frac{\mg_{n,m,\Delta}^{[\mF]}}{\mg_{n,m,\Delta}}.
\]
The expression of $\mg_{n,m,\Delta}^{[\mF]}$ has been derived in Theorem~\ref{th:distinguished_degree_constraints-symbolic},
while Lemma~\ref{lem:total_weight} provides the expression of $\mg_{n,m,\Delta}$.
\end{proof}

Corollary~\ref{th:deg_constraints_mean} is useful to
find the threshold for the emergence of a given multigraph $F$
as a subgraph in a random $(n,m,\Delta)$-multigraph,
\ie the rate at which $m = m(n)$ should go to infinity with~$n$
so that a random $(n,m,\Delta)$-multigraph
typically contains a bounded, but positive number of $F$-subgraphs.
In order to derive more information on the limit law of the number of $F$-subgraphs in a random $(n,m,\Delta)$-multigraph,
we extend the notion of patchworks (see Sections~\ref{sec:patchworkdef} and~\ref{sec:patchworks}) to $\Delta$-multigraphs. 
To take into account the degrees, the generating function of patchworks of a multigraph family $\mF$ becomes
\[
	\Patch{\mF}(z,w,\vy,u) =
    \sum_{P \in \mPatch{\mF}}
    u^{|P|}
    \bigg(
    \prod_{v \in V(P)}
    y_{\deg(v)}
    \bigg)
    \frac{w^{m(P)}}{2^{m(P)} m(P)!}
    \frac{z^{n(P)}}{n(P)!},
\]
where $\deg(v)$ denotes the degree of the vertex $v$ in the patchwork $P$.
We also denote by $\mPatch{\mF}(u)$ the set of $\mF$-patchworks,
where the weight of a patchwork $P$ is defined as
\[
	\omega(P) = u^{|P|}.
\]

\begin{corollary}[Probability for number of subgraphs, degree constraints, multigraphs]
\label{th:deg_constraints_exact_subgraphs}
Given a multigraph family $\mF$,
the probability for a random $(n,m,\Delta)$-multigraph
to contain exactly $t$ $\mF$-subgraphs is equal to
\[
	\frac{[u^t] \mg_{n,m,\Delta}^{[\mPatch{\mF}(u-1)]}}{\mg_{n,m,\Delta}},
\]
where the family $\mPatch{\mF}(u-1)$ is defined just before the corollary,
the term $\mg_{n,m,\Delta}^{[\mPatch{\mF}(u-1)]}$
is expressed in Theorem~\ref{th:distinguished_degree_constraints-symbolic},
and the value of $\mg_{n,m,\Delta}$ is provided by Lemma~\ref{lem:total_weight}.
\end{corollary}

\begin{proof}
This is a direct application of Proposition~\ref{th:number_of_subgraphs_general_weights}
to multigraphs weighted according to their degrees.
\end{proof}

Corollaries~\ref{th:deg_constraints_mean} and~\ref{th:deg_constraints_exact_subgraphs}
reduce the study of the number of $\mF$-subgraphs
in a random $(n,m,\Delta)$-multigraph
to the asymptotic analysis of $\mg_{n,m,\Delta}^{[\mF]}$,
the total weight of $(n,m,\Delta)$-multigraphs
where one $\mF$-subgraph is distinguished.
This asymptotic analysis widely depends on the weight vector $\vdelta$
and the analytic properties of its generating function $\Delta(x)$.
In the rest of Section~\ref{sec:results-degree-constraints} we start with 
the cases in which either only finitely many weights are nonzero 
or the weights do not increase too fast (this is Lemma~\ref{th:distinguished_degree_constraints-asymptotic_infinite-set} followed by its applications). In Section~\ref{sec:failed-assumptions} some cases in which the conditions of Lemma~\ref{th:distinguished_degree_constraints-asymptotic_infinite-set} fail are discussed.

	\subsection{Finite set of nonzero weights}
    \label{sec:finite-set-weights}

When there is only a finite number of nonzero weights,
$\Delta(x)$ is a polynomial
and the $\Delta$-multigraphs have bounded degrees.
More precisely, since the sum of the degrees is twice the number of edges,
for any $(n,m,\Delta)$-multigraph with nonzero weight,
\[
	n \min(\support(\Delta)) \leq 2m \leq n \max(\support(\Delta)),
\]
which implies that $2m/n$ is bounded
and in $[\min(\support(\Delta)), \max(\support(\Delta))]$.
When $2m/n$ reaches one of those bounds,
then the multigraph is regular
(all its vertices have the same degree).
In particular, if there is only one nonzero weight,
\emph{i.e.}, $\Delta(x)$ is a monomial,
then $2m/n$ is equal to the corresponding degree
and the multigraph is regular.
The analysis of subgraphs in regular multigraphs
has been achieved previously
(we refer the reader to Bollob\'as~\cite{Bollobas80} and to McKay, Wormald and Wysocka~\cite{McKayWormaldWysocka}),
so we omit this case here
and consider $2m/n$ having a limit in $]\min(\support(\Delta)), \max(\support(\Delta))[$.

\begin{lemma}
\label{th:distinguished_degree_constraints-asymptotic_infinite-set}
Consider a weight generating function $\Delta(x)$
of infinite radius of convergence and support $\support(\Delta)$. 
Let the integers $n$ and $m := m(n)$ tend to infinity and assume
\begin{enumerate}
\renewcommand{\labelenumi}{(C\arabic{enumi})}
\item \label{assump:1}
\[
	\inf_{n\ge 0} \frac{2m}{n} > \min(\support(\Delta)), \quad \sup_{n\ge 0} \frac{2m}{n} < \max(\support(\Delta)).
\]
\end{enumerate}
Let $\chi := \chi(n)$ denote the unique positive solution of
\[
 \frac{\chi\Delta'(\chi)}{\Delta(\chi)}=\frac{2m}n.
\]
Consider a multigraph family $\mF$ with generating function $F(z,w,\vy)$,
and let $\mg_{n,m,\Delta}$ denote the total weight of $(n,m,\Delta)$-multigraphs.
Moreover, set
\begin{equation}\label{eq:def_g}
G_n(z,x,t) = F\left(n z,\frac{1}{2 m t^2}, \left( \frac{(x\chi)^d \Delta^{(d)}(x\chi)}{\Delta(x \chi)}\right)_{d\ge 0} \right)
\end{equation}
and for a given $\eps >0$,
\[
L_n(z,x,t)=
\begin{cases}
\displaystyle \frac{G_n(z,x,t)}{G_n(1,1,1)} & \text{ if } t\ge \eps, \\[4mm]
\displaystyle \frac{G_n(z,x,\eps)}{G_n(1,1,1)} & \text{ if } 0\le t\le \eps.
\end{cases}
\]
Assume moreover:
\begin{enumerate}
\renewcommand{\labelenumi}{(C\arabic{enumi})}
\stepcounter{enumi}
\item \label{assump:2} $L_n(z,x,t)$ converges to some limit $L(z,x,t)$, as $n\to\infty$ and uniformly 
for $|z|=1$, $|x|=1$ such that $\arg x\le \eps$ and $t$ in any compact sub-interval of $[0,\infty)$. 
Moreover, $L_n(z,x,t)$ is uniformly bounded for $z$ and $t$ as above and $|x|=1$. 
\item \label{assump:3} For all $0<\eps<\frac\pi2,$ we have 
\[
\int_{-\pi}^\pi \frac{\Delta\left(\chi e^{i\theta}\right)^n}{e^{2mi\theta}}\de \theta \sim \int_{-\eps}^{\eps} \frac{\Delta\left(\chi e^{i\theta}\right)^n}{e^{2mi\theta}}\de \theta, 
\]
as $n\to\infty$. 
\end{enumerate}
Then the number $\mg_{n,m,\Delta}^{[\mF]}$ of $(n,m,\Delta)$-multigraphs with one distinguished 
$\mF$-subgraph is asymptotically equal to 
$
    \mg_{n,m,\Delta} \cdot G_n(1,1,1).
$
\end{lemma}

\begin{rem}
The conditions of this lemma may appear to be very strong requirements.
Nevertheless there are some ``natural'' assumptions on the weight sequence under which the conditions of Lemma~\ref{th:distinguished_degree_constraints-asymptotic_infinite-set} are satisfied: Assume that the weights are bounded in the sense 
\begin{equation}\label{eq:bound_weight}
	\delta_j\le K^j
\end{equation}
for some positive constant $K$ and that we do not have any periodicities, \ie the set of allowed degrees is not a subset of some lattice $r+\ell\mathbb N$ with integers $r\ge 0$ and $\ell\ge 2$. 
The bound~\eqref{eq:bound_weight} guarantees that $\Delta(x)$ has infinite radius of convergence. If we additionally assume that $\ell\mathbb N\subseteq \support(\Delta)$ for some positive integer $\ell$, then $\Delta(x)$ is a linear combination of terms of the form $e^{Kx+2\pi i/\ell}$ plus smaller order terms and thus behaves essentially like $e^{Kx}$. 
This implies that locally around $x=1$, all terms of the form $x^d\Delta^{(d)}(x\chi)/\Delta(x\chi) $ converge uniformly even if $\chi$ tends to infinity. 
Consequently, Condition~(C\ref{assump:2}) is satisfied. 
To show Condition~(C\ref{assump:3}), we first observe that the aperiodicity condition guarantees that the only region that matters is $x$ near 1. That means that we need the uniform convergence of Condition~(C\ref{assump:2}) only for $x$ near 1. 
Furthermore, note that $e^{Kx}$ is Hayman-admissible (\emph{cf.} \cite{Hay56}), which essentially means that the saddle point method directly applies. 
Hayman's proof uses certain estimates of Cauchy-like integrals. 
One of them straight-forwardly implies Condition~(C\ref{assump:3}), if there is no periodicity. 

In case of periodicity, we can still apply Lemma~\ref{th:distinguished_degree_constraints-asymptotic_infinite-set} in a slightly modified form. 
Then there are more, say $s$, crucial regions, one near 1 and the others distributed along a regular $s$-gon. 
Conditions~(C\ref{assump:2}) and~(C\ref{assump:3}) have to be extended accordingly. 
Cases where periodicity occurs will be discussed in Section~\ref{sec:failed-assumptions}.
\end{rem}

\begin{proof}
Recall Equation~\eqref{eq:distinguished_degree_constraints}:
\begin{equation*} 
\mg_{n,m,\Delta}^{[\mF]} = n! 2^m m! [z^n w^m]
    \sum_{j \geq 0}
    (2j)! \, [x^{2j}]
    F \left(z,w,\bar{\partial} \Delta (x) \right)
    e^{z \Delta(x)}
    \frac{w^j}{2^j j!}.
\end{equation*} 
Using the classical formula (moments of even order of the standard normal distribution)
\[
	\frac{1}{\sqrt{2\pi}}
    \int_{-\infty}^{\infty} t^{2j} e^{-t^2/2} \de t= \sqrt\frac2\pi \int_0^\infty t^{2j} e^{-t^2/2} \de t 
     = \frac{(2j)!}{2^j j!},
\]
we rewrite the sum on the right-hand side of \eqref{eq:distinguished_degree_constraints} as 
\begin{align*}  
\sum_{j \geq 0} \frac{(2j)!}{2^j j!} 
    [x^{2j}] F(z,w,\bar{\partial} \Delta (x))
    e^{z \Delta(x)} w^j 
&=\sum_{j \geq 0}
    \sqrt\frac2\pi
    \int_0^{\infty}
    w^j t^{2j} e^{-t^2/2} \de t
    [x^{2j}] F(z,w,\bar\partial \Delta(x))
    e^{z \Delta(x)} 
    \\
&= \sqrt\frac2\pi
    \int_0^{\infty}
    \left( \sum_{j\ge 0} [x^{2j}]F(z,w,\bar\partial\Delta(x)) (t\sqrt{w})^{2j} \right) 
    e^{z \Delta(t\sqrt{w})}
    e^{-t^2/2} \de t    
	\\
&= \sqrt\frac2\pi
    \int_0^{\infty}
    F(z,w,(\bar\partial\Delta)(t\sqrt{w}))
    e^{z \Delta(t\sqrt{w})}
    e^{-t^2/2} \de t,
\end{align*} 
where interchanging the sum and the integral is licit, 
because for sufficiently small $x$ the series converges uniformly. For the evaluation of the sum in
the last step we used the fact that the series $F$ has only even powers of $x$. This follows from
the fact that $x$ counts the number of half-edges in the graph, which is always even (\emph{cf.}
Theorem~\ref{th:distinguished_degree_constraints-symbolic}).
%
To obtain $\mg_{n,m,\Delta}^{[\mF]}$, we extract the coefficient $[z^n w^m]$ of the previous expression and multiply by $n! 2^m m!$:
\begin{align}\label{eq:dist_finite_coeff}
	\mg_{n,m,\Delta}^{[F]} =
	n! \frac{2^{m+\frac12}}{\sqrt\pi} m!  [z^n w^m]
    \int_0^{\infty}
    F(z,w,(\bar\partial\Delta)(t\sqrt{w}))
    e^{z \Delta(t \sqrt{w})}
    e^{-t^2/2} \de t.
\end{align}
Evaluation of the coefficients of $z^n$ and $w^m$ is based on the saddle-point method. This leads to 
\[
\mg_{n,m,\Delta}^{[\mF]} =
\frac{n! 2^{m+\frac12} m!}{\sqrt\pi(2\pi i)^2} \int_0^{\infty} 
\int_{|w|=g(t)} \int_{|z|=n/\Delta(t\sqrt w)} \frac{F(z,w,(\bar\partial\Delta)(t\sqrt{w})) e^{z
\Delta(t \sqrt{w})}}{z^{n+1} w^{m+1}}\de z\de w\, e^{-t^2/2} \de t, 
\]
where 
\[
g(t) = \begin{cases} \chi^2/t^2 & \text{ if } t\ge \eps, \\
\chi^2/\eps^2 & \text{ if } 0\le t\le\eps,
\end{cases}
\]
and by Condition (C1) the saddle-point $\chi$ remains in a compact subinterval of $]0, + \infty[$.
To simplify we apply successively the following changes of variables: 
\[
	z \to \frac{n z}{\Delta(\sqrt{w} t)},
    \qquad
    w = (x \chi /t)^2,
    \qquad
    t \to \sqrt{2m} t.
\]
The expression becomes 
\begin{align} 
\frac{\mg_{n,m,\Delta}^{[\mF]}}{G(1,1,1)}&=\frac{n!2^{m+\frac12} m! (2m)^{m+\frac12}}{n^n \chi^{2m}\sqrt\pi (2\pi i)^2}
\left(\int_\eps^{\infty} 
\int_{|x|=1} \int_{|z|=1} \frac{L_n(z,x,t) e^{nz} \Delta(x\chi)^n}{z^{n+1} x^{2m+1}}\de z\de x\,
t^{2m}e^{-mt^2} \de t\right. \nonumber \\
&\left.\qquad+ \eps^{2m} \int_0^\eps 
\int_{|x|=1} \int_{|z|=1} \frac{L_n(z,x,t) e^{nz} \Delta(x\chi)^n}{z^{n+1} x^{2m+1}}\de z\de x \,
e^{-mt^2} \de t\right). \label{Gl_1}
\end{align} 

Since $L_n(z,x,t)$ is uniformly bounded, by Condition~(C\ref{assump:3}) truncating the circle $|x|=1$ to $|x|=1, \arg(x)\le \eps$ in \eqref{Gl_1} causes an asymptotically negligible error. Next, we start with the evaluation of the inner integral. Note that $\arg(x)\le \eps$ and thus Condition~(C\ref{assump:2}) applies: $L_n(z,x,t)$ converges uniformly to $L(z,x,t)$. Thus we obtain 
\begin{equation}\label{Gl_2}
\frac1{2\pi i}\int_{|z|=1} \frac{L_n(z,x,t) e^{nz}}{z^{n+1}} \de z \sim \frac1{2\pi i}\int_{|z|=1}
\frac{L(z,x,t) e^{nz}}{z^{n+1}} \de z \sim L(1,x,t) [z^n]e^{nz} = L(1,x,t) \frac{n^n}{n!}, 
\end{equation}
where we used Lemma~\ref{th:large_powers} in the last asymptotic estimate. 

The next integral is 
\begin{equation}\label{Gl_3}
\frac1{2\pi i} \int_{|x|=1} \frac{L(1,x,t)\Delta(x\chi)^n}{x^{2m+1}}\de x \sim \frac1{2\pi i}
\int_{|x|=1, \arg(x)\le \eps} \frac{L(1,x,t)\Delta(x\chi)^n}{x^{2m+1}}\de x
\sim L(1,1,t)[x^{2m}]\Delta(x\chi)^n
\end{equation}
which follows from the analyticity of $L(1,x,t)$, Condition~(C\ref{assump:3}) and the Taylor expansion  
$L(1,x,t)\sim L(1,1,t)+L_x(1,1,t)(x-1)$ together with $x-1=\bigO(\eps)$.  

For the last integral we use the Laplace method (see \cite[Theorem~B.7]{FS09}). 
First observe that the function $t\mapsto L(1,1,t)$ is positive and decreasing on the positive real line and
that $t^{2m} e^{-mt^2}$ attains its maximum at $t=1$. Thus, \cite[Theorem~B.7]{FS09} tells us that 
\begin{align}
\int_\eps^\infty L(1,1,t)t^{2m} e^{-mt^2} \de t &\sim L(1,1,1) \int_\eps^\infty t^{2m} e^{-mt^2} \de
t \nonumber \\ 
&\sim L(1,1,1) \int_0^\infty t^{2m} e^{-mt^2} \de t =
L(1,1,1) \sqrt{\frac\pi 2} \frac{1}{(2m)^{m+\frac12}}\frac{(2m)!}{2^m m!}, \label{Gl_4}
\end{align}
as $m\to\infty$ and for $\eps<1$. Finally, we have to show that  
\[
\int_0^\eps L(1,1,t)t^{2m} e^{-mt^2} \de t
\]
is negligible. Note that $L_n(1,1,t)$ is constant on the integration interval, and so its
limit $L(1,1,t)$ is constant as well. The exponential is bounded by 1 and thus the integral is
$O(\eps^{2m})$. Since $\int_0^\infty t^{2m} e^{-mt^2} \de t=\Theta(\sqrt m e^{-m})$, $\eps^{2m}$
is negligible for $\eps<1/\sqrt{e}$. 

Collecting the contributions from \eqref{Gl_2}--\eqref{Gl_4} and inserting them into \eqref{Gl_1} gives finally 
\[
\frac{\mg_{n,m,\Delta}^{[\mF]}}{G(1,1,1)}\sim (2m)! [x^{2m}]\Delta(x)^n = \mg_{n,m,\Delta}
\]
which concludes the proof.
\end{proof}

In particular, if all multigraphs from $\mF$
have maximal degree $d$ greater than $\max(\support(\Delta(x)))$,
then $\Delta^{(d)}(\chi)$ vanishes,
so the lemma claims that $\mg_{n,m,\Delta}^{[\mF]}$ tends to $0$.
Indeed, $\mg_{n,m,\Delta}^{[\mF]}$ vanishes,
because the maximum degree of a subgraph of a multigraph $G$
is never greater than the maximum degree of $G$.

\begin{theorem}[Finite set of weights, trees and cycles]
\label{th:finite_trees_cycles}
Consider a polynomial $\Delta(x)$,
integers $n$ and $m := m(n)$
going to infinity such that
$2m/n$ has a limit in $]\min(\support(\Delta(x))), \max(\support(\Delta(x)))[$,
the unique positive solution $\chi$ of the equation
\[
	\frac{\chi \Delta'(\chi)}{\Delta(\chi)} =
    \lim_{n \to \infty} \frac{2m}{n},
\]
and a random $(n,m,\Delta)$-multigraph $G$.
We also assume that $\Delta(x)$ is aperiodic, meaning that
\[
	\gcd \{d_1 - d_2\ |\ \delta_{d_1} \neq 0, \delta_{d_2} \neq 0\} = 1
\]
(the periodic case is treated in Section~\ref{sec:failed-regular}).
\begin{itemize}
\item[i)]
Any connected multigraph
that is neither a tree nor a unicycle
is asymptotically almost surely not a subgraph of $G$.
\item[ii)]
Let $T$ be a tree with $k$ vertices and $\aut{T}$ automorphisms.
Then the expected number of $T$-subgraphs in $G$ is asymptotically equal to
\[
	\frac{1}{\aut{T}}
    \frac{n^k}{(2m)^{k-1}}
    \prod_{v \in V(T)} \frac{\chi^{\deg(v)} \Delta^{(\deg(v))}(\chi)}{\Delta(\chi)}.
\]
\item[iii)]
Denoting by $C_\ell$ a cycle of length $\ell$, 
then $G[C_\ell]$ follows asymptotically a Poisson law of mean
\[
    \frac{1}{2\ell}
    \left( \frac{n}{m}
    \frac{\chi^2 \Delta''(\chi)}{\Delta(\chi)}
    \right)^\ell.
\]
\end{itemize}
\end{theorem}

\begin{proof}

i) Any connected multigraph $H$ that is neither a tree nor a cycle
has a essential density greater than~$1$.
Let $F$ denote a subgraph of $H$ of density greater than~$1$.
Then, when $2m/n$ and $\chi$ have finite limits,
\[
	F \left(n, \frac{1}{2m}, \left( \frac{\chi^d \Delta^{(d)}(\chi)}{\Delta(\chi)} \right)_{d \geq 0} \right)
\]
tends to zero.
According to Proposition~\ref{th:general_weight_mean},
the expected number of $F$-subgraphs in a random $(n,m,\Delta)$-multigraph is
\[
	\frac{\mg_{n,m,\Delta}^{[F]}}{\mg_{n,m,\Delta}}.
\]
Condition~(C\ref{assump:1}) of Lemma~\ref{th:distinguished_degree_constraints-asymptotic_infinite-set} is satisfied.
The polynomial $F(z,w,\vy)$ satisfies Condition~(C\ref{assump:2}).
Finally, Condition~(C\ref{assump:3}) is satisfied by the aperiodic polynomial $\Delta(x)$,
by application of the Daffodil Lemma \cite{FS09}.
Applying Lemma~\ref{th:distinguished_degree_constraints-asymptotic_infinite-set},
we obtain
\[
	\mg_{n,m,\Delta}^{[F]} \sim 
    F \left(n, \frac{1}{2m}, \left( \frac{\chi^d \Delta^{(d)}(\chi)}{\Delta(\chi)} \right)_{d \geq 0} \right)
    \mg_{n,m,\Delta},
\]
so the expected number of $F$-subgraph in a random $(n,m,\Delta)$-multigraph $G$ tends to $0$.
Therefore, with high probability, $G$ contains no $H$-subgraph.

ii) The second result of the theorem is also a direct consequence of Lemma~\ref{th:distinguished_degree_constraints-asymptotic_infinite-set}.

iii) The last result is proven by application of Proposition~\ref{th:general_weights_second_asymptotics}.
We chose for $\alpha_n$ the function
\[
	\alpha_n(F) = F \left(n, \frac{1}{2m}, \left( \frac{\chi^d \Delta^{(d)}(\chi)}{\Delta(\chi)} \right)_{d \geq 0} \right).
\]
When $C$ is a cycle, its number of vertices is the same as its number of edges, so
\[
	\alpha_n(C) = C \left(n, \frac{1}{2m}, \left( \frac{\chi^d \Delta^{(d)}(\chi)}{\Delta(\chi)} \right)_{d \geq 0} \right)
\]
has a finite limit (as $m$ is proportional to $n$, and $\chi$ has a finite limit). 
Consequently, the first condition of Proposition~\ref{th:general_weights_second_asymptotics}
is satisfied.
The second condition is obtained by application of Lemma~\ref{th:distinguished_degree_constraints-asymptotic_infinite-set},
and the third condition is trivial.
Thus, the number of cycles of length $\ell$ follows a Poisson limit law
of parameter
\[
	\alpha_n(C) =
    \frac{1}{2\ell} \left( \frac{n}{m}
    \frac{\chi^2 \Delta''(\chi)}{\Delta(\chi)}
    \right)^\ell.\qedhere
\]
\end{proof}

\begin{rem}
The case of $p$-regular multigraphs, which is a trivial example of periodic $\Delta$, can be dealt with in a simple way. Notice that $\Delta(x) = \frac{x^p}{p!}$, and thus we have
\[
	\frac{x^d \Delta^{(d)}(x)}{\Delta(x)} = p(p-1)\dots(p-d+1) = \frac{p!}{(p-d)!}
\]
for $0 \le d\le p$
These quantities do not depend on $x$.
Hence, the expression of $G(z,x,t)$ becomes
\[
	G(z,x,t) = F\left(nz, \frac{1}{2mt^2}, \left(\mu_d\right)_{d\ge 0}\right)  \quad 
    \textrm{ where } \mu_d = \begin{cases} 
    	p(p-1)\dots(p-d+1), &\forall 0\le d \le p,\\
        0, &\forall d > p.
    \end{cases}
\]
Condition~(C\ref{assump:2}) of Lemma~\ref{th:distinguished_degree_constraints-asymptotic_infinite-set} is trivially satisfied, while Conditions~(C\ref{assump:1}) and (C\ref{assump:3}) are no longer needed. It suffices to notice that the total number of half-edges must be a multiple of $p$ and that
\[
	[x^{pn}] \Delta^n(x) = \frac{1}{(p!)^n}.
\]

\end{rem}

\begin{proposition}
Let $\mF$ be a multigraph family with generating function $F(z,w,\vy)$, and let $\mathrm{Reg}_{p,n}$ denote the total weight of $p$-regular multigraphs with $n$ vertices (and $pn/2$ edges). The expected number of copies of elements of $\mF$ is asymptotically equal to
\[
	\frac{\mathrm{Reg}^{[\mF]}_{p,n}}{\mathrm{Reg}_{p,n}} \sim F\left(n, \frac{1}{np}, \left(\mu_d)_{d\ge 0}\right)\right) \quad 
    \textrm{ where } \mu_d = \begin{cases} 
    	p(p-1)\dots(p-d+1), &\text{ if } 0\le d \le p,\\
        0, &\text{ if } d > p.
    \end{cases}
\]
\end{proposition}

	\subsection{Infinite set of nonzero weights}
    \label{sec:infinite-set-weights}

We now turn to the case when there is an infinite number of nonzero weights, but these weights do not grow ``too fast''; in particular we assume that their generating function has infinite radius of convergence, so Lemma~\ref{th:distinguished_degree_constraints-asymptotic_infinite-set} is applicable. 
Theorem~\ref{th:limit_multi_weights} below is the equivalent of Theorem~\ref{th:limit_law_simple}, in the case of degree constraints.

\begin{theorem}[Poisson law, weights, strictly balanced, multigraphs]
\label{th:limit_multi_weights}
Consider a weight generating function $\Delta(x)$
with nonnegative coefficients,
infinite radius of convergence,
that is not a polynomial,
and such that for all $d \geq 0$,
the ratio $\Delta^{(d)}(x) / \Delta(x)$
has a finite positive limit when $x$ goes to infinity.
Let $F$ be a strictly balanced multigraph of density $d$.
Consider an integer sequence $m = \exactbigO(n^{2 - 1/d})$,
and let $\chi := \chi(n)$ denote the unique positive solution of the relation
\[
	\frac{\chi \Delta'(\chi)}{\Delta(\chi)} = \frac{2m}{n}.
\]
Suppose that $\Delta(x)$ satisfies Condition~(C\ref{assump:3})
of Lemma~\ref{th:distinguished_degree_constraints-asymptotic_infinite-set}.
Then the number of $F$-subgraphs in a random $(n,m,\Delta)$-multigraph
has a limit Poisson law of parameter
\[
	\lim_{n \to \infty}
    F \left(n, \frac{1}{2m}, 
    	\left( \frac{\chi^d \Delta^{(d)}(\chi)}{\Delta(\chi)} \right)_{d \geq 0} \right).
\]
\end{theorem}

\begin{proof}
The only strictly balanced multigraphs of density $1$ are the cycles.
Their number follows a Poisson law,
and the proof is the same as for Theorem~\ref{th:finite_trees_cycles}.
So we focus on the case of a strictly balanced multigraph $F$
with density greater than $1$, so $2m/n$ tends to infinity.
The relation
\[
	\frac{\chi \Delta'(\chi)}{\Delta(\chi)} = \frac{2m}{n}
\]
implies that $\chi$ tends to infinity with $n$. By assumption, $\frac{\Delta'(\chi)}{\Delta(\chi)}$ has a finite limit $\ell$,
which gives
\[
	\chi \sim \frac{2 m}{\ell n}.
\]
This implies that
\[
	\alpha_n(F) =
    F \left(n, \frac{1}{2m}, 
    	\left( \frac{\chi^d \Delta^{(d)}(\chi)}{\Delta(\chi)} \right)_{d \geq 0} \right)
\]
has a finite limit.
Also, for any multigraph $H$ of density higher than $d$,
\[
	\alpha_n(H) = \exactbigO \left(
    \frac{n^{n(H)}}{(2m)^{m(H)}} 
    \prod_{v \in V(H)} \frac{\chi^{\deg(v)} \Delta^{(d)}(\chi)}{\Delta(\chi)}
    \right) =
    \smallo(1), 
\]
because $m = \exactbigO(n^{2 - 1/d})$, $\frac{m(H)}{n(H)} > d$ and $\sum_{v \in V(H)} \deg(v) = 2 m(H)$.
By application of Lemma~\ref{th:distinguished_degree_constraints-asymptotic_infinite-set} 
the second assumption of Proposition~\ref{th:general_weights_second_asymptotics} is satisfied,
so the number of $F$-subgraphs in a random $(n,m,\Delta)$-multigraph
follows asymptotically a Poisson law of parameter $\alpha_n(F)$.
\end{proof}

\section{Extensions: when the conditions of Lemma~\ref{th:distinguished_degree_constraints-asymptotic_infinite-set} do not hold}
\label{sec:failed-assumptions}

We consider in this section an example for each case when one of the three conditions of  Lemma~\ref{th:distinguished_degree_constraints-asymptotic_infinite-set} fails.

\subsection{When Condition (C1) does not hold: $\frac{2m}{n} \to 0$, threshold for trees}
\label{sec:failed-trees}

We look here at what happens when  $m/n \rightarrow 0$, and study the appearance of trees.

Given a family of degrees defined by its generating function $\Delta(x)= \sum\limits_{d\ge 0} \delta_d \frac{x^d}{d!}$, set
\[
	\mu(j) := \min\{d : d\ge j, \delta_d>0\}, \textrm{ for any }j\ge 0.
\]
In particular, we have $\mu(0) = 0$, \ie $\delta_0 > 0$. Otherwise there would be no isolated vertex, and the number of edges would grow at least linearly in the number of vertices ($2m \geq n$), contradicting  $m/n \rightarrow 0$.
Define  also, for a tree~$T$,
\[
\gamma_\Delta (T) := \sum\limits_{v\in T} \mu(\deg(v)).
\]

\begin{theorem}[Trees in sparse multigraphs]
\label{th:weights-trees}
Let $T$ be a given tree with $k$ vertices. Suppose that $\Delta$ satisfies Condition~(C\ref{assump:3}), and consider a random $(n,m,\Delta)$-multigraph $G$ with $m=\Theta(n\epsilon_n)$, where $\epsilon_n \to 0$. Then
\[
	\esp(G[T]) = \Theta\left( n \epsilon_n^{ -(k -1) + \frac{\gamma_\Delta (T) }{\mu(1)}}\right).
\]
Thus, when $\epsilon_n = \smallo  ( n^{\lambda} )$ with 
$ \lambda = \frac  {\mu(1)}  {(k-1)\mu(1) - \gamma_\Delta (T) }$, 
$G[T] = 0$ almost surely.
\end{theorem}

\begin{proof}
As $\frac{2m}{n}$ tends to $0$, so does the solution $\chi$ of the saddle-point equation~\eqref{eq:x-boltzmann}, and we can use the saddle-point equation to estimate the order of magnitude of the saddle-point and of the derivatives of $\Delta$ at $\chi$:
\[
	\Delta(\chi) = \delta_0 + \smallo(1) = \Theta(1), \quad
    \Delta^{(j)}(\chi) = \delta_{\mu(j)} \frac{\chi^{\mu(j)-j}}{\mu(j)!} + \smallo(\chi^{\mu(j)-j}) = \Theta(\chi^{\mu(j)-j}),\quad  \text{ for all } j\ge 1.
\]
This implies 
\[
	\frac{\chi \Delta'(\chi)}{\Delta(\chi)} = \Theta(\chi^{\mu(1)}) = \frac{2m}{n} = \Theta(\epsilon_n).
\]
Then we obtain
\[
	\chi = \Theta\left(\epsilon_n^\frac{1}{\mu(1)}\right), 
    \quad \frac{\chi^j \Delta^{(j)}(\chi)}{\Delta(\chi)} = \Theta\left(\epsilon_n^\frac{\mu(j)}{\mu(1)}\right).
\]
%
Although Condition (C1) of Lemma~\ref{th:distinguished_degree_constraints-asymptotic_infinite-set} does not hold, as long as Condition (C3) is valid the proof can be adapted, using the fact that $T(z,w,\textbf{y})$ is simply a monomial in this case:
\[
	T(z,w,\textbf{y}) = \frac{z^{k+1}}{(k+1)!} \frac{w^{k}}{2^k k!} \prod_{v\in V(T)} y_{\deg(v)}, \quad\textrm{ where } k= |E(T)|.
\]
First, we directly get Equation~\eqref{eq:dist_finite_coeff}:
\[
	\mg_{n,m,\Delta}^{[T]} =
	n! \frac{2^{m+\frac12}}{\sqrt\pi} m!  [z^n w^m]
    \int_0^{\infty}
    T(z,w,(\bar\partial\Delta)(t\sqrt{w}))
    e^{z \Delta(t \sqrt{w})}
    e^{-t^2/2} \de t.
\]
Applying the changes of variables
\[
	z \rightarrow \frac{nz}{\Delta(t\sqrt{w})}, \quad w \rightarrow \frac{x^2}{t^2}, \quad t \rightarrow \sqrt{2m} t
\]
we get
\[
	\mg_{n,m,\Delta}^{[T]} =
	\frac{n!}{n^n} \frac{ (4m)^{m+\frac12}}{\sqrt\pi} m!  [z^n x^{2m}]
    \int_0^{\infty}
    T\left(nz,\frac{1}{2mt^2},\left( \frac{x^d \Delta^{(d)}(x)}{\Delta(x)} \right)_d \right)
    \Delta^n(x) e^{nz}
    t^{2m} e^{-mt^2} \de t.
\]
Similarly, setting
\[
	G_n(z,x,t) := T\left(nz,\frac{1}{2mt^2},\left( \frac{x^d \Delta^{(d)}(x)}{\Delta(x)}  \right)_d \right), \quad L_n(z,x,t) := \frac{G_n(z,x,t)}{G_n(1,\chi,1)},
\]
we obtain
\[
	\mg_{n,m,\Delta}^{[T]} = 
     G_n(1,\chi,1) \cdot
	\frac{n!}{n^n} \frac{ (4m)^{m+\frac12}}{\sqrt\pi} m! [z^n x^{2m}]
    \int_0^{\infty}
    L_n(z,x,t)
    \Delta^n(x) e^{nz}
    t^{2m} e^{-mt^2} \de t.
\]
Since $L_n(z,x,t)$ is a monomial, we can apply Laplace's method for $t=1$ and obtain 
\[
	\int_0^\infty L_n(z,x,t) t^{2m} e^{-mt^2} \mathrm dt 
    \sim_{n \to \infty} L_n(z,x,1)\int_0^\infty t^{2m} e^{-mt^2} \mathrm dt 
    = L_n(z,x,1) \sqrt{\frac{\pi}{2}} \frac{1}{(2m)^{m+\frac12}} \frac{(2m)!}{2^m m!}, 
\]
uniformly for $|z|=1$ and $x$ in some neighbourhood of the origin. 
Then we extract the coefficient in $z$ with a simple saddle-point method at $z=1$:
\[
	[z^n] L_n(z,x,1) e^{nz} \sim L_n(1,x,1) \frac{n^n}{n!}.
\]
We are left with the extraction of the coefficient in $x$:
\[
	[x^{2m}] L_n(1,x,1) \Delta^n(x) 
    = \frac{1}{2i\pi} \int_{|x|=\chi} L_n(1,x,1) \frac{\Delta^n(x)}{x^{2m+1}} \mathrm dx 
    = \frac{1}{2\pi \chi^{2m}} \int_{-\pi}^\pi L_n(1,\chi e^{i\theta},1) \frac{\Delta^n(\chi e^{i\theta})}{e^{2im\theta}} \mathrm d\theta.
\]
By Condition~(C\ref{assump:3}) and the analyticity of $L_n$, the main contribution originates from the region around 0. We thus have
\[
	[x^{2m}] L_n(1,x,1) \Delta^n(x) \sim L(1,\chi,1) [x^{2m}] \Delta^n(x) = L(1,\chi,1) \frac{[x^{2m}] \Delta^n(x)}{(2m)!}.
\]
Collecting everything, we obtain the asymptotics of the expected number of occurrences of $T$:
\[
	\esp({G[T]}) = \frac{\mg_{n,m,\Delta}^{[F]}}{\mg_{n,m,\Delta}} \sim G_n(1,\chi,1).
\]
Substituting $m = n\epsilon_n$, we finally have
\[
	\esp({G[T]}) = G_n(1,\chi,1) = T \left( n,\frac{1}{2m}, \left( \frac{ \chi^d \Delta^{(d)} (\chi) }{\Delta(\chi)} \right)_d \right) 
   = \Theta\left( n \epsilon_n^{ -(k -1) +  \frac{\gamma_\Delta (T)}{\mu(1)}}\right).
\]
Again, the random variable $G[T]$ exhibits a threshold behaviour at the value
\[ 
	\epsilon_n := \Theta\left(n^{\frac{\mu(1)}{(k-1)\mu(1) - \gamma_\Delta (T) }}\right) \to 0. \qedhere
\]
\end{proof}

Notice that we recover the threshold obtained in Corollary~\ref{cor:essential_density_multigraph} when all degrees are allowed, which implies that $\mu(j) = j$, for $j\ge 0$, and that $\gamma_\Delta (T) $ is simply twice the number of edges of~$T$:
\[
	\epsilon_n := n^{\alpha_0 -1} \Rightarrow \alpha_0 = 1 + \frac{1}{(k-1) - \gamma_\Delta (T) } = 1 - \frac{1}{k-1} = 2 - \frac{1}{d(T)}.
\]

\subsection{When Condition (C2) does not hold: Power law}
\label{sec:power-law}

This subsection is devoted to the study of weighted multigraphs where the distribution of the vertex degrees in a random multigraph $G$ follows a power law with parameter $\beta >1$, \emph{i.e.}:
\[
\exists C > 0, \forall d \ge 1, \forall v\in V(G):\ \mathbb{P}( \deg(v) = d ) = C d^{-\beta}.
\]

Many real-world networks exhibit a power-law degree distribution, usually with a parameter $\beta$ located between 2 and 3. The degree distribution being fixed, the ratio $2m/n$ is then determined by $\beta$. In our model, this corresponds to weights $\delta_d = d^{-\beta} d!$ in the expression of $\Delta$ and to evaluating it at $x=1$. We obtain
\[
	\Delta_\beta(x) := \sum_{d \ge 1} d^{-\beta} x^d \quad\text{ and }\quad \proba(\deg(\Gamma^{\mathcal{B}}_{x=1} V_{\Delta_\beta}) = d) = d^{-\beta}.
\]
Notice that $\Delta_\beta$ now has finite radius of convergence equal to 1, and at $x=1$ we have
\[
	\Delta_\beta(1) = \zeta(\beta), \quad \Delta_\beta'(1) = \zeta(\beta-1),
\]
where $\zeta(s) := \sum\limits_{n\ge 1} \frac{1}{n^s}$ is the Riemann zeta function.
In this case, with the saddle-point located at $\chi =1$, the saddle-point equation has a particular shape:
\[ 
	\frac{\Delta_\beta'(1)}{\Delta_\beta(1)} = \frac{\zeta(\beta-1)}{\zeta(\beta)} = \frac{2m}{n}.
\]
Hence the limit of $2m/n$ is prescribed by the parameter $\beta$ of the power law.

Lemma~\ref{th:distinguished_degree_constraints-asymptotic_infinite-set} does not apply here as Condition~(C\ref{assump:2}) is not satisfied: $L_n(z,x,t)$ is in general not analytic around $x=1$, because of the higher derivatives of $\Delta_\beta$ involved in the formula.

Nonetheless, most of the steps are still valid. After a similar reasoning to extract the coefficient $[z^n]$ and estimate the integral with respect to $t$, one obtains
\[
	\mg_{n,m,\Delta_\beta}^{[F]} \sim 
    (2m)!  [x^{2m}]
    F\left( n, \frac{1}{2m},\left( \frac{x^d \Delta_\beta^{(d)}(x)}{\Delta_\beta(x)} \right)_{d\ge 0} \right)
    \Delta_\beta(x)^n.
\]
Let us study a simple family $F$ with a unique element, with $n_F$ vertices and $m_F$ edges. Then
\[
	F(z,w,{\bf y}) = \frac{1}{\aut{F}}\frac{z^{n_F}}{n_F!} \frac{w^{m_F}}{2^{m_F} m_F!} \prod_{v\in F} y_{\deg(v)}.
\]
Replacing $F$ by its expression in the above formula yields
\[
	\mg_{n,m,\Delta_\beta}^{[F]} \sim 
    \frac{(2m)!  n^{n_F}}{(2m)^{m_F} n_F! 2^{m_F} m_F!} [x^{2m}]
    \prod_{v\in F} \frac{x^{\deg(v)} \Delta_\beta^{(\deg(v))}(x)}{\Delta_\beta(x)}
    \Delta_\beta(x)^n.
\]
A standard saddle-point technique cannot be applied directly here, as there is a conflict between the saddle-point at $x=1$ with the large power $\Delta_\beta(x)^n$ and the singularities of $\Delta_\beta^{(j)}(x), j\ge 2$.

To circumvent this problem, we use a reformulation of Lagrange inversion (\emph{cf.} Lemma~\ref{lem:Lagrange} in the appendix).
In our setting, after some rewriting, we obtain
\[
	[x^{2m}]\prod_{v\in F}\frac{x^{\deg(v)} \Delta_\beta^{(\deg(v))}(x)}{\Delta_\beta(x)} \Delta_\beta(x)^n = [x^{2m-n}]\prod_{v\in F}\frac{x^{\deg(v)} \Delta_\beta^{(\deg(v))}(x)}{\Delta_\beta(x)} \left(\left(\frac{\Delta_\beta(x)}{x}\right)^{\frac{n}{2m-n}}\right)^{2m-n} .
\]
Lemma~\ref{lem:Lagrange} can now be applied with
\[
	\Phi(x) := \left(\frac{\Delta_\beta(x)}{x}\right)^{\frac{n}{2m-n}}, \quad H(x) := \prod_{v\in F}\frac{x^{\deg(v)} \Delta_\beta^{(\deg(v))}(x)}{\Delta_\beta(x)}
\]
and with $T(z)$ defined implicitly by the equation $T(z) = z \Phi(T(z))$.
We get
\begin{equation}
	[x^{2m}]\prod_{v\in F}\frac{x^{\deg(v)} \Delta_\beta^{(\deg(v))}(x)}{\Delta_\beta(x)} \Delta_\beta(x)^n = [z^{2m-n}] \frac{z T'(z)}{T(z)} \prod_{v\in F}\frac{T(z)^{\deg(v)} \Delta_\beta^{(\deg(v))}(T(z))}{\Delta_\beta(T(z))}.
\label{eq:res-Lagrange}
\end{equation}
The right part of this equation is now amenable to singularity analysis.

\paragraph{Asymptotic expansions of $\Delta_\beta$, $\Phi$ and $T$ when $2< \beta <3$.}

The function $\Delta(x)$ is a simple case of a polylogarithm, whose asymptotics around $x=1$ is known (see Flajolet and Sedgewick~\cite[Theorem VI.7, p.~408]{FS09}):
\[
	\Delta_\beta(x) \sim \Gamma(1-\beta) (-\log(x))^{\beta -1} + \sum_{j\ge 0} \frac{(-1)^j}{j!} \zeta(\beta -j) (-\log(x))^j.
\]
When $2<\beta <3$, one has more precisely
\[
	\Delta_\beta(x) = \zeta(\beta) - \zeta(\beta-1)(1-x) + \Gamma(1-\beta)(1-x)^{\beta -1} + (\zeta(\beta -2) - \zeta(\beta -1)) \frac{(1-x)^2}{2} + O(1-x)^\beta ,
\]
which leads to the following asymptotics for $\Phi(x)$:
\[
	\Phi(x) = \zeta(\beta)^{\frac{n}{2m-n}} + \frac{n}{2m-n} \frac{\zeta(\beta)-\zeta(\beta-1)}{\zeta(\beta)} \zeta(\beta)^{\frac{n}{2m-n}}  (1-x) + \frac{n\Gamma(1-\beta)}{2m-n} \zeta(\beta)^{\frac{n}{2m-n}-1}(1-x)^{\beta -1} + O(1-x)^2 .
\]
Notice that
\[
	\frac{n}{2m-n} = \frac{1}{2m/n - 1} = \frac{1}{\Delta'_\beta(1)/\Delta_\beta(1) -1} = \frac{1}{\zeta(\beta -1)/\zeta(\beta) -1} = \frac{\zeta(\beta)}{\zeta(\beta -1) -\zeta(\beta)}.
\]
Finally one gets
\[ 
	\Phi(x) = \zeta(\beta)^{\frac{n}{2m-n}} \left(x + \frac{\Gamma(1-\beta)}{\zeta(\beta-1) - \zeta(\beta)} (1-x)^{\beta -1}\right) + O(1-x)^2 .
\]
To obtain the asymptotics of $T(z)=z\Phi(T(z))$, one might want to use a classical Lagrange inversion. Here the saddle-point lies on the circle of convergence of $\Phi(x)$, as they are both equal to~1. Fortunately, with this remarkable expansion of $\Phi(x)$, it falls into the setting presented by Flajolet and Sedgewick in~\cite[Section VI.18, p.~407]{FS09},  yielding directly the $\Delta$-analyticity of $T(z)$ as well as its expansion:
\begin{equation}
	T(z) = 1- \tau^{\frac{1}{\beta-1}}\left(1- \frac{z}{\rho}\right)^{\frac{1}{\beta-1}} + \smallo\left(1-\frac{z}{\rho}\right)^{\frac{1}{\beta-1}}, 
    \textrm{where } \tau := \frac{\zeta(\beta-1) - \zeta(\beta)}{\Gamma(1-\beta)}, \textrm{ and } \rho := \zeta(\beta)^{-\frac{n}{2m-n}}.
    \label{eq:dev-T}
\end{equation}
Moreover, we have 
\[
	T'(z) =
  \frac{\tau^{\frac{1}{\beta-1}}}{\rho (\beta-1)}
  \left( 1 - \frac{z}{\rho } \right) ^ { \frac{1}{\beta-1}-1} 
  + o \left( 1 - \frac{z}{\rho } \right)^{\frac{1}{\beta-1}-1}.
\]
We do not pursue the computation of the asymptotics in the general case, as it becomes somewhat intricate, but consider below the specific case of cycles.

\subsubsection*{Cycles of fixed length} 
When $F$ is a cycle of length $\ell, (\ell \ge 3)$, we have
\[	
	F(z,w,\vy) = \frac{1}{2\ell} \frac{z^\ell}{\ell!} \frac{w^\ell}{2^\ell \ell!} 
\; y_2^\ell .
\]

\begin{theorem}[Cycles in power-law multigraphs]
\label{th:power-law-cycles}
When the weights follow a power law of parameter~$\beta \; (2<\beta<3)$, the  expected number of $\ell$-cycles in a random $(n,m,\Delta_\beta)$-multigraph is equal to
\[
	\frac{\mg_{n,m,\Delta_\beta}^{[F]}}{\mg_{n,m,\Delta_\beta}} 
    \sim 
    \kappa_{\beta,\ell} \, n^{\frac{3-\beta}{\beta-1}\ell},
\]
where $\kappa_{\beta,\ell}$ is independent of $n$ and can be explicitly computed.
\end{theorem}

\begin{proof}
	We want to extract the coefficient $[z^{2m-n}]$ in the following expression, coming from \eqref{eq:res-Lagrange}:
\[
	\frac{z T'(z)}{T(z)} \left( \frac{T(z)^2 \Delta_\beta''(T(z))}{\Delta_\beta(T(z))} \right)^\ell.
\]
Using~\cite[Theorem VI.7, p.~408]{FS09} again, we first obtain the asymptotics of~$\Delta_\beta''$ around $x=1$: 
\begin{eqnarray*}
	\Delta_\beta''(x) &=& 
    \frac{\Delta_{\beta-2}(x) - \Delta_{\beta-1}(x)}{x^2} 
    \\ &=& 
    \Gamma(3-\beta)(-\log(x))^{\beta-3} + \smallo(\log^{\beta-3}(x)),
\end{eqnarray*}
where the second equation holds because $\beta-3 <0$, all the other powers of $\xi$ being positive.
With the expansion~\eqref{eq:dev-T} for $T(z)$, we get
\begin{align*}
	\Delta_\beta(T(z)) &= \zeta(\beta) + \smallo(1) ;
    \\
    \Delta''_\beta(T(z)) &= 
    \Gamma(3-\beta)\left( \tau\left(1-\frac{z}{\rho}\right) \right)^{\frac{\beta-3}{\beta-1}} +\smallo\left(1-\frac{z}{\rho} \right)^{\frac{\beta-3}{\beta-1}}.
\end{align*}
We are in the critical composition scheme (cf. [p.~412]\cite{FS09}): The value of $T$ at its singularity is equal to the singularity of $\Delta$,
$$
T(\rho) = 1.
$$
We now collect every expansion and obtain 
\[
	\frac{z T'(z)}{T(z)} \left( \frac{T(z)^2 \Delta_\beta''(T(z))}{\Delta_\beta(T(z))} \right)^\ell \sim z\frac{\tau^{\frac{\beta-3}{\beta-1}\ell+\frac{1}{\beta-1}} \Gamma(3-\beta)^\ell}{\rho(\beta-1)\zeta(\beta)^\ell} \left( 1- \frac{z}{\rho}\right)^{ \frac{\beta -3}{\beta -1}\ell + \frac{1}{\beta -1} -1}.
\]
Applying the transfer lemma presented in Flajolet and Sedgewick~\cite{FS09}, we get finally
\begin{align*}
	[z^{2m-n}] \frac{z T'(z)}{T(z)} \left( \frac{T(z)^2 \Delta_\beta''(T(z))}{\Delta_\beta(T(z))} \right)^\ell 
    &\sim [z^{2m-n-1}] \frac{\tau^{\frac{(\beta-3)\ell +1}{\beta-1}} \Gamma(3-\beta)^\ell}{\rho(\beta-1)\zeta(\beta)^\ell} \left( 1- \frac{z}{\rho}\right)^{\frac{(\beta-3)\ell - \beta +2}{\beta -1}} \\
    &\sim \frac{\tau^{\frac{(\beta-3)\ell +1}{\beta-1}} \Gamma(3-\beta)^\ell}{(\beta-1)\zeta(\beta)^\ell} \frac{(2m-n-1)^{-\frac{(\beta-3)\ell -\beta +2}{\beta-1}-1}}{\Gamma\left(-\frac{(\beta-3)\ell - \beta +2}{\beta-1}\right)} \rho^{-(2m-n)} \\
    &\sim 
    \frac{\tau^{\frac{(\beta-3)\ell +1}{\beta-1}} \Gamma(3-\beta)^\ell }{(\beta-1)\zeta(\beta)^\ell \, \Gamma\left(\frac{(3-\beta)\ell + \beta -2}{\beta-1}\right)} 
    (2m-n)^{\frac{(3-\beta)\ell -1}{\beta-1}} \,
    \zeta(\beta)^{n}.
 \end{align*}
Now, from the saddle-point equation the term $(2m-n)$ can also be written as
\[
\left(
\frac{\zeta (\beta-1) }{ \zeta (\beta)} -1
\right) \, 
n.
\]
We obtain the asymptotics of $\mg_{n,m,\Delta_\beta}^{[F]}$:
\[
	\mg_{n,m,\Delta_\beta}^{[F]} \sim \frac{(2m)!  n^{\ell}}{(2m)^{\ell} \ell!^2 2^{\ell}} 
    \frac{\tau^{\frac{(\beta-3)\ell +1}{\beta-1}} \Gamma(3-\beta)^\ell }{(\beta-1)\zeta(\beta)^\ell \, \Gamma\left(\frac{(3-\beta)\ell + \beta -2}{\beta-1}\right)} 
    \left(\frac{\zeta (\beta-1) }{ \zeta (\beta)} -1\right)^{\frac{(3-\beta)\ell -1}{\beta-1}}
    n^{\frac{(3-\beta)\ell -1}{\beta-1}} \,
    \zeta(\beta)^{n}.
\]
Similarly, we can obtain the asymptotics of $\mg_{n,m,\Delta_\beta}$ (by simply setting $\ell=0$ in the above formula):
\[
	\mg_{n,m,\Delta_\beta} \sim (2m)!
    \frac{\tau^{\frac{1}{\beta-1}} }{(\beta-1) \, \Gamma\left(\frac{\beta -2}{\beta-1}\right)} 
    \left(\frac{\zeta (\beta-1) }{ \zeta (\beta)} -1\right)^{-\frac{1}{\beta-1}}
    n^{-\frac{1}{\beta-1}} \,
    \zeta(\beta)^{n}.
\]
Setting 
\begin{align*}
\kappa_\beta 
:= & \frac{\tau^{\frac{\beta-3}{\beta-1}\ell}}{ \ell!^2 2^{\ell}}  
\left( \frac{\zeta (\beta-1) }{ \zeta (\beta)} -1 \right)^{\frac{3-\beta}{\beta-1}\ell}
\,
\frac{ \Gamma(\frac{\beta-2}{\beta-1}) }{\Gamma\left(\frac{(3-\beta)\ell+\beta -2}{\beta-1}\right)} \left(\frac{\Gamma(3-\beta)}{\zeta(\beta-1)}\right)^\ell \\
= & \frac{1}{\ell!^2 2^{\ell}}  
\left( \frac{\zeta (\beta) }{ \Gamma (1-\beta)} \right)^{\frac{\beta-3}{\beta-1}\ell}
\,
\frac{ \Gamma(\frac{\beta-2}{\beta-1}) }{\Gamma\left(\frac{(3-\beta)\ell+\beta -2}{\beta-1}\right)} \left(\frac{\Gamma(3-\beta)}{\zeta(\beta-1)}\right)^\ell \\
\end{align*}
finishes the proof of Theorem~\ref{th:power-law-cycles}.
\end{proof}

\subsection{When Condition (C3) does not hold: Eulerian multigraphs and graphs with periodic degree sequence}
\label{sec:failed-regular}

When there exist integers $r \geq 0$ and $p \geq 2$
such that the index of any nonzero coefficient of $\Delta$
is equal to $r$ modulo $p$, the function $\Delta$ is said to be \emph{periodic}.
When $p$ is the largest integer satisfying this property,
then $\Delta$ is said to be \emph{$p$-periodic}.
This is in particular the case of Eulerian graphs,
where each vertex has an even degree, so
\[
	\Delta(x) = \cosh(x).
\]
According to the Daffodil Lemma (\cite[Lemma IV.1]{FS09}),
there is a function $\Omega$ analytic at $0$
and aperiodic (\ie its largest period is $1$) such that
\[
	\Delta(x) = x^r \Omega(x^p).
\]
Condition~(C\ref{assump:2}), which states that the main contribution of the integral of $\Delta$
on a circle centered at the origin comes from the part close to the real axis,
no longer holds.
Indeed, for any positive real value $\chi$ and $p$th root $\rho$ of unity, we have
\[
	|\Delta(\rho \chi)| = |(\rho \chi)^r \Omega(\rho^p \chi^p)| = \Delta(\chi).
\]
The number of edges of an $(n,m,\Delta)$-multigraph must then satisfy
\[
	2 m = \sum_{v \in V(G)} \deg(v) = n r + p \sum_{v \in V(G)} \frac{\deg(v) - r}{p},
\]
so $2 m - n r$ must be divisible by $p$.

According to Lemma~\ref{lem:total_weight},
the total weight of $(n,m,\Delta)$-multigraphs is
\[
	\mg_{n,m,\Delta} = (2m)! [x^{2m}] \Delta(x)^n = (2m)! [x^{2m}] x^{n r} \Omega(x^p)^n.
\]
This implies that $p$ must divide $2m - n r$
(otherwise, there are no $(n,m,\Delta)$-multigraphs, so the total weight is $0$), and
\[
	\mg_{n,m,\Delta} = (2m)! [x^{\frac{2m - n r}{p}}] \Omega(x)^n.
\]

Likewise, for each integer $d$, the function $x^d \Delta^{(d)}(x)$ is also $p$-periodic,
so there is a function $\Omega_d$ such that
\[
	x^d \Delta^{(d)}(x) = x^r \Omega_d(x^p).
\]
The equation
\[
	\mg_{n,m,\Delta}^{[\mF]} =
    \frac{n! 2^{m+\frac{1}{2}} m! (2m)^{m + \frac{1}{2}}}{n^n \sqrt{\pi}}
    \int_{t=0}^{\infty}
    [z^n x^{2m}]
    F \left(n z, \frac{1}{2m t^2}, 
    	\left( \frac{x^d \Delta^{(d)}(x)}{\Delta(x)} \right)_{d \geq 0} \right)
    e^{n z} \Delta(x)^n t^{2m} e^{-m t^2} \de t
\]
becomes
\[
	\mg_{n,m,\Delta}^{[\mF]} =
    \frac{n! 2^{m+\frac{1}{2}} m! (2m)^{m + \frac{1}{2}}}{n^n \sqrt{\pi}}
    \int_{t=0}^{\infty}
    [z^n x^{\frac{2m - n r}{p}}]
    F \left(n z, \frac{1}{2m t^2}, 
    	\left( \frac{\Omega_d(x)}{\Omega(x)} \right)_{d \geq 0} \right)
    e^{n z} \Omega(x)^n t^{2m} e^{-m t^2} \de t.
\]
The rest of the analysis of subgraphs in $(n,m,\Delta)$-multigraphs
is then the same, and we conclude that
any multigraph $F$ with generating function $F(z,w,\vy)$ satisfies the following assertion: If $m := m(n)$ tends to infinity with $n$ in such a way that
\begin{itemize}
\item
$p$ divides $2 m - n r$,
\item
and the term 
\[
	E_n = 
    F \left(n, \frac{1}{2m}, 
    	\left( \frac{\Omega_d(\chi)}{\Omega(\chi)} \right)_{d \geq 0} \right), 
\]

where $\chi := \chi(n)$ is the unique positive solution of the equation 
\[
	\frac{\chi \Omega'(\chi)}{\Omega(\chi)} = \frac{2 m - n r}{p n},
\]
has a finite limit, as $n$ tends to infinity, 
\end{itemize}
then the limit of $E_n$ is equal to the asymptotic expected number of $F$-subgraphs
in a random $(n,m,\Delta)$-multigraph.

When $2m / n$ tends to infinity, then $\chi$ tends to infinity with $n$.
In that case, there are other more subtle ways for Condition~(C\ref{assump:2}) to fail.
For example, $\Delta$ could be aperiodic,
but may be close to a periodic function
on circles of large radius.
For example, Condition~(C\ref{assump:2}) is not valid for
\[
	\Delta(x) = \sinh(x) + 1
\]
for any $2m/n$ tending to infinity.
Those weights correspond to multigraphs
where each vertex has either odd degree or degree $0$.
Those other examples require a case-by-case refinement of our analysis
to determine the contributions on the circle of radius $\chi$.

\section{Perspectives, future works and limitations}

We presented an approach via analytic combinatorics to the problem of counting subgraphs in various models of random graphs. The approach gives precise expressions for subgraph counts and distributional results on the number of subgraphs belonging to an \emph{a priori} given family of graphs. The notion of patchworks was the crucial concept to keep track of vertices and edges even if several copies of subgraphs overlap. 

This approach can certainly be enhanced into several directions. One is the treatment of strictly balanced \emph{induced} subgraphs. There are already several examples in the literature treating ``fixed'' as well as induced subgraphs. 

Another direction is weakening the balancedness property. In order to cope with overlappings, it proved convenient to require that the subgraphs under consideration must be strictly balanced. It seems hard to drop this requirement completely, but the condition can be weakened to barely balanced subgraphs (for the definition see the end of Section~\ref{sec_Models}). Certainly, some additional technicalities have to be overcome, but we are confident that the methodology presented in this paper can be adapted to cover subgraph counts for barely balanced graphs. This problem is the topic of a forthcoming paper, see~\cite{CG19}. 

What is currently not covered by our approach are random graph with other constraints than constraints on their node degrees. It seems not out of reach to extend the approach to graph families whith properties being  amenable to specifications with generating functions. Examples for such graph families are probably connected graphs or multi-partite graphs. 

Except for varying the graph family we may think of modifying the stage in the random graph process. It is well known that there are many copies of a given subgraph if we are beyond the threshold. So far, we have no results on what \emph{exactly} happens beyond the threshold, \emph{i.e.}, asymptotics for subgraph counts, Gaussian limiting distributions, \emph{etc.} As mentioned at the end of the last section, certain subtleties may arise which seem to require a case-by-case study. 

What seems currently out of reach is the counting problem for graph families defined via topological properties like forbidding certain minors. In this case, we would need the series of minor patchworks. Our approach enables us to count graphs with a distinguished occurrence of a minor, but this does not allow us to count graphs with at least one occurrence, since the graphs with many occurrences will appear more often than the ones with few occurrences. The resulting distribution is therefore not uniform. 

Another problem we currently cannot solve is the case of infinite families of subgraphs. For instance, when the subgraph family is the set of all cycles (of arbitrary length) and we are asking for the number of $\mathcal F$-free graphs, then the resulting counting problem is that for the number of trees on $n$ vertices. Although our approach is quite powerful and general, it is currently not allowing to cover this simple problem. Since having a certain minor or not may also be described in our setting by means of an infinity (multi)graph family, an extension to infinite family seem particularly desirable.

\bigskip
\paragraph{Acknowledgment.} We thank a referee for careful reading which lead to substantial improvements of the original paper.

\bibliographystyle{plain}
\bibliography{biblio}

\newpage

\setcounter{section}{0}
\renewcommand{\thesection}{A\arabic{section}}

		\section{Appendix: analytic and combinatorial tools}
        \label{sec:tools}

\subsection{Lagrange inversion}
The Lagrange inversion theorem allows us to obtain the coefficients of a function defined by an implicit equation; see for example~\cite[p.~732-33]{FS09}.
Here we use it in the following form:
\begin{lemma}\label{lem:Lagrange}
Let $\Phi(t), H(t)$ be formal power series, such that: $\Phi(0) \neq 0$. Let $T(z)$ be a function implicitly defined by: $T(z) = z \Phi (T(z))$. Then, for all $n \ge 0$:
\[ 
	[t^n] H(t) \Phi(t)^n = [z^n] \frac{z T'(z)}{T(z)} H(T(z)).
\]
\end{lemma}

\begin{proof}
It follows from a simple change of variable $t=T(z)$, \emph{i.e.} $z=t/\Phi(t)$:
\[
	[z^n] \frac{z T'(z)}{T(z)} H(T(z)) = \frac{1}{2i\pi} \oint \frac{z T'(z)}{T(z)} H(T(z)) \frac{dz}{z^{n+1}} = \frac{1}{2i\pi} \oint \frac{H(t)}{t} \frac{dt}{(t/\Phi(t))^n} = [t^n] H(t) \Phi(t)^n. \qedhere
\]
\end{proof}

\subsection{The saddle-point heuristic}
\label{sec:saddle-point-heuristic}


We use in our proofs a simple case of the Laplace method
(see e.g.\ the book  of Pemantle and Wilson~\cite{PW13}).

\begin{lemma}
[Laplace]
Consider two entire functions $A(t)$ and $B(t)$, where $B(t)$ is a positive function that reaches its unique maximum on the real open interval $I$ at a point $r$, and $A(r) \neq 0$.
Then we have
\[
	\int_I A(t) B(t)^n \de t
    \sim 
    A(r) \int_I B(t)^n \de t
\]
whenever the integral is well defined.
\end{lemma}

The saddle-point method is a technique to compute
the asymptotics of the coefficients of a generating function.
The coefficient extraction is written as a Cauchy integral,
on which a Laplace method is applied.
There exist many variations of this technique.
We will use here the following lemma,
which is a particular case of Theorem~VIII.8 from \cite{FS09}.

\begin{lemma}[Saddle point] 
\label{th:large_powers}
Consider two entire functions $A(z)$ and $B(z)$,
and a sequence of integers $N(n)$ such that $N(n)/n$
has a positive finite limit $\lambda$.
Assume there exists a unique positive solution $r$ to the equation
\[
	\frac{r B'(r)}{B(r)} = \lambda,
\]
such that $A(r) \neq 0$ and $\left( \frac{r B'(r)}{B(r)} \right)' \neq 0$. Then
\[
	[z^N] A(z) B(z)^n \sim A(r) [z^N] B(z)^n.
\]
\end{lemma}

\end{document}